\documentclass[final,hidelinks,onefignum,onetabnum]{siamart251216}
\input{shared_info}

\hypersetup{
  colorlinks=false, pdfborder={0 0 1}, linkbordercolor={1 0 0}, citebordercolor={0 1 0} 
}

\begin{document}

\maketitle

\begin{abstract}
    This paper develops specular differentiation for mappings between real normed vector spaces, extending the classical notions of G\^ateaux and Fr\'echet differentiability.
    These notions apply to a broad class of nonsmooth functions, including functions for which certain other generalized derivatives do not exist.
    We investigate their fundamental theoretical properties and establish a Quasi-Mean Value Theorem for real-valued functions.
    Finally, we prove that the specular G\^ateaux derivative of a semi-convex functional, whenever it exists at an interior point of its effective domain, determines a global quadratic lower bound and belongs to its Fr\'echet subdifferential.
\end{abstract}



\section{Introduction}

In this paper, we extend \emph{specular differentiation}, previously studied in Euclidean spaces \cite{2026a_Jung,2022a_Jung}, to mappings between real normed vector spaces.
This approach is distinct from other notions of generalized differentiation, including weak derivatives \cite{2003_Adams}, Dini derivatives \cite{1994_Bruckner_BOOK}, Clarke derivatives \cite{1990_Clarke_BOOK}, symmetric derivatives \cite{1967_Aull}, and subdifferentials \cite{1970_Rockafellar_BOOK}.
One motivation for studying specular differentiation is that numerical methods based on it can outperform classical methods in certain settings; see \cite{2026b_Jung,2026d_Jung,py_1.3.3}.
Motivated by these developments, this paper provides a theoretical foundation for further research on specular differentiation in both finite- and infinite-dimensional normed spaces.

For nonsmooth functions, specular differentiability can provide a useful alternative when classical differentiability is unavailable or too restrictive.
The definition of specular differentiation is motivated by the law of reflection in geometrical optics.
In fact, we will show that, in inner product spaces, the tangent vector in the specular sense is orthogonal to the sum of the forward and backward unit tangent vectors.
Moreover, specular differentiation is amenable to numerical computation.
For a convex functional, its G\^ateaux derivative in the specular sense, whenever it exists, belongs to its subdifferential.
This provides a way to select and compute a subgradient through specular differentiation.
Thus, specular differentiation offers a potential tool for problems in which the differential structure of a nonsmooth function is otherwise difficult to exploit.

The original concepts of specular differentiation of real-valued functions on $\mathbb{R}^n$ were introduced in \cite{2023_Jung}.
A recent study \cite{2026a_Jung} revised the original definition in $\mathbb{R}$, compared specular and symmetric derivatives,
and provided nontrivial examples.
For the one-dimensional Quasi-Mean Value Theorem in the specular sense, see \cite[Thm.~3.3]{2026a_Jung}.
For the Mean Value Theorem, we refer the reader to \cite[Thm.~3.2.7]{2013_Drabek_BOOK} for classical differentiation and \cite{2007_Schirotzek_BOOK} for generalized differentiation.
Jung and Oh \cite{2022a_Jung} revised and extended the earlier theory in \cite{2023_Jung} in light of \cite{2026a_Jung}, introducing the notion of \emph{regular specular differentiation} when one-sided derivatives exist and are finite.
For further properties of specular tangent hyperplanes and normal vectors in finite dimensions, see \cite[Sect.~3]{2022a_Jung}.
Therefore, this paper builds on \cite{2026a_Jung,2022a_Jung}.
The author explicitly excludes the definitions and results of \cite{2023_Jung} from the theoretical framework of the present paper and will not rely on them in subsequent work on specular differentiation.

The main results of this paper are summarized as follows.
First, we define specular differentiation and investigate the relationships between classical and specular differentiation in normed vector spaces; see \cref{prop:G_implies_spG,prop:sF_implies_F,prop:spF_implies_spG}.
For locally Lipschitz mappings defined on open subsets of $\mathbb{R}^n$, we also prove that specular G\^ateaux and specular Fr\'echet differentiability are equivalent; see \cref{thm:lipschitz_sG_sF_Rn}.
Second, we give a geometric characterization of specular tangent vectors; see \cref{thm:specular_tangent_bisector}.
Third, we establish the Quasi-Mean Value Theorem in the specular sense in normed vector spaces; see \cref{thm:QMVT_nvs_real_valued}.
Fourth, we show that the specular G\^ateaux derivative of a semi-convex functional determines a global quadratic lower bound and belongs to its Fr\'echet subdifferential; see \cref{thm:sGd_is_subd}.

\subsection{Definitions and notations}

Throughout this paper, we employ the following notation.
Let $X$ and $Y$ be vector spaces over $\mathbb{R}$ equipped with norms $\| \cdot \|_X$ and $\| \cdot \|_Y$, respectively.
Define a norm on $X \times Y$ by
\begin{equation}    \label{def:prod_norm}
    \left\| (x, y) \right\|_{X \times Y} := \left( \left\| x \right\|_X^2 + \left\| y \right\|_Y^2 \right)^{\frac{1}{2}}
\end{equation}
for $(x, y) \in X \times Y$.

Let $\mathcal{L}(X; Y)$ denote the space of continuous linear operators $\ell : X \to Y$, equipped with the operator norm 
\begin{displaymath}
    \left\| \ell \right\|_{\mathcal{L}(X; Y)} := \sup_{\left\| x \right\|_X = 1} \left\| \ell(x) \right\|_Y
\end{displaymath}
for $\ell \in \mathcal{L}(X; Y)$.
In particular, we denote by $X^{\ast} := \mathcal{L}(X; \mathbb{R})$ the \emph{dual space} of $X$.
The \emph{duality pairing} between $X$ and $X^{\ast}$ is the bilinear mapping $\left\langle \, \cdot \,, \, \cdot \,  \right\rangle : X^{\ast} \times X \to \mathbb{R}$, defined by $\left\langle \ell, x \right\rangle := \ell(x)$ for $\ell \in X^{\ast}$ and $x \in X$.

The following function is central to the definition of specular differentiation.
The geometric motivation for this function is deferred to \cref{sec:spdiff}.
\begin{definition}
    Let $Y$ be a normed vector space over $\mathbb{R}$ equipped with norm $\| \cdot \|_Y$. 
    For $\sigma > 0$, define the function $\mathcal{A}_{\sigma}^Y : Y \times Y \times (0, \infty) \to Y$ by
    \begin{displaymath}
        \mathcal{A}_{\sigma}^Y (y, z, h) 
        := 
        \frac{z \left\| (\sigma h, y) \right\|_{\mathbb{R} \times Y} + y \left\| (\sigma h, z) \right\|_{\mathbb{R} \times Y}}{h\left\| (\sigma h, y) \right\|_{\mathbb{R} \times Y} + h\left\| (\sigma h, z) \right\|_{\mathbb{R} \times Y}}.
    \end{displaymath}    
    When $\sigma=1$, we simply write $\mathcal{A}^Y := \mathcal{A}_1^Y$.
    When $Y = \mathbb{R}$, we simply write $\mathcal{A}_{\sigma} := \mathcal{A}_{\sigma}^{\mathbb{R}}$.
\end{definition}

Note that 
\begin{equation} \label{eq:A_subscript}
    \mathcal{A}_{\sigma}^Y (y, z, h) 
    =
    \sigma \mathcal{A}^Y \left( \frac{y}{\sigma}, \frac{z}{\sigma}, h \right) 
    =
    \sigma \mathcal{A}^Y (y, z, \sigma h)
\end{equation}
and
\begin{equation} \label{eq:A_multiplied_arguments}
    \mathcal{A}^Y(\lambda y, \lambda z, \lambda h)
    =
    \mathcal{A}^Y(y, z, h)
\end{equation}
for all $y, z \in Y$ and $h, \lambda > 0$.

We introduce the basic definitions and notation for \emph{specular directional derivatives}, \emph{specular G\^ateaux derivatives}, and \emph{specular Fr\'echet differentials}.
The LaTeX macros for all new symbols introduced in this paper are available in \cite{py_1.3.3}.

\begin{definition}    \label{def: specular Gateaux derivative}
    Let $X$ and $Y$ be normed vector spaces over $\mathbb{R}$ equipped with norms $\| \cdot \|_X$ and $\| \cdot \|_Y$, respectively. 
    Let $f : \Omega \to Y$ be a function, where $\Omega$ is an open subset of $X$.
    We define the \emph{specular directional derivative} $\partial_{v}^{\sd}f(x)$ of $f$ at $x \in \Omega$ in the direction of $v \in X \setminus \left\{ 0 \right\}$ as 
    \begin{equation}   \label{eq:sdd}
        \partial_{v}^{\sd}f(x) 
        :=
        \lim_{h \searrow 0} \mathcal{A}_{\| v \|_X}^Y \bigl( f(x+hv)-f(x), f(x)-f(x-hv), h \bigr),
    \end{equation}
    provided that the limit exists in $Y$.
    For $v=0$, we set $\partial_0^{\sd}f(x):=0$.

    If $X = \mathbb{R}^n$ and $Y = \mathbb{R}$ are equipped with Euclidean norms, then we write
    \begin{displaymath}
        \frac{\eth}{\eth x_i} f(x) := \partial_{x_i}^{\sd}f(x) := \partial_{e_i}^{\sd}f(x)
    \end{displaymath}
    and call it the \emph{specular partial derivative} of $f$ at $x = (x_1, x_2, \ldots, x_n) \in \Omega$ with respect to the variable $x_i$, where $\left\{ e_i \right\}_{i=1}^n$ is the standard basis of $\mathbb{R}^n$. 
    
    If $X = \mathbb{R}$, write 
    \begin{displaymath}    
        \frac{\sGd}{\sGd x} f(x) := f^{\sd}(x) := \partial_1^{\sd} f(x) \in Y,
    \end{displaymath}
    which we call the \emph{specular derivative} of $f$ at $x \in \Omega$.
\end{definition}

Specular directional derivatives generalize directional derivatives; see \cref{prop:dd_implies_sdd}.
In $\mathbb{R}$, the specular derivative is expressed in terms of one-sided derivatives under suitable assumptions \cite{2026a_Jung}.
Here, we generalize some of these results to mappings between real normed vector spaces.

For a function $f : \Omega \to Y$, where $\Omega$ is an open subset of $X$, define the \emph{one-sided directional derivatives} of $f$ at $x \in \Omega$ in the direction of $v \in X$ by
\begin{displaymath}
    \partial_v^+ f(x)
    := \lim_{h \searrow 0} \frac{f(x + hv) - f(x)}{h}
    \qquad\text{and}\qquad
    \partial_v^- f(x)
    := \lim_{h \searrow 0} \frac{f(x) - f(x - hv)}{h},
\end{displaymath}
whenever the respective limits exist in $Y$.

If $X = \mathbb{R}$, write $\partial^+ f(x) := \partial_1^+ f(x)$ and $\partial^- f(x) := \partial_1^- f(x)$, which we call the \emph{right} and \emph{left derivatives} of $f$ at $x$, respectively.

In the case where $Y = \mathbb{R}$ is equipped with the absolute value norm, we also allow these limits to equal $\pm\infty$.
Thus, the one-sided directional derivatives may take values in the \emph{extended real line} $\overline{\mathbb{R}} := \mathbb{R} \cup \{-\infty,\infty\}$.

Under suitable assumptions, specular directional derivatives can be written in terms of one-sided directional derivatives; see \cref{thm:repr_sdd}.

\begin{definition}
    Let $X$ and $Y$ be normed vector spaces over $\mathbb{R}$ equipped with norms $\| \cdot \|_X$ and $\| \cdot \|_Y$, respectively. 
    Let $f : \Omega \to Y$ be a function, where $\Omega$ is an open subset of $X$.
    Fix $x \in \Omega$ and $v \in X$.
    We say that $f$ is \emph{regularly specularly differentiable} at $x$ in the direction $v$ if both $\partial_v^+f(x)$ and
    $\partial_v^-f(x)$ exist in $Y$.
    In this case, $\partial_v^{\sd}f(x)$ is called the \emph{regular specular directional derivative} of $f$ at $x$ in the direction $v$.
\end{definition}

We now introduce specular generalizations of G\^ateaux and Fr\'echet differentiation for mappings between real normed vector spaces.
For the classical notions of G\^ateaux and Fr\'echet differentiation, we refer the reader to \cite{2010_Borwein_BOOK,2025_Ciarlet,2013_Drabek_BOOK,2007_Schirotzek_BOOK}.

\begin{definition}
    Let $X$ and $Y$ be normed vector spaces over $\mathbb{R}$ equipped with norms $\| \cdot \|_X$ and $\| \cdot \|_Y$, respectively. 
    Let $f : \Omega \to Y$ be a function, where $\Omega$ is an open subset of $X$.
    Fix $x \in \Omega$.

    If $\partial_{v}^{\sd}f(x)$ exists for all $v \in X$ and there exists $\ell \in \mathcal{L}(X; Y)$ such that $\ell(v) = \partial_{v}^{\sd}f(x)$ for all $v \in X$, then we say $f$ is \emph{specularly G\^ateaux differentiable} at $x$.
    If such an operator $\ell$ exists, we call it the \emph{specular G\^ateaux derivative} of $f$ at $x$ and write $\sGd f(x) := \ell$.

    We say $f$ is \emph{specularly Fr\'echet differentiable} at $x$ if there exists $\ell \in \mathcal{L}(X; Y)$ such that 
    \begin{equation}   \label{eq:sFd}      
        \lim_{\|w\|_X \to 0} \left\| \, \mathcal{A}^Y \! \bigl( f(x+w) - f(x), f(x) - f(x-w), {\left\| w \right\|_X} \bigr) - \frac{\ell(w)}{\left\| w \right\|_X} \right\|_Y
        =
        0,
    \end{equation}
    or equivalently, 
    \begin{displaymath}
        \left\| \|w\|_X \mathcal{A}^Y \! \bigl( f(x+w)-f(x), f(x)-f(x-w), \|w\|_X \bigr) - \ell(w) \right\|_Y
        = 
        o(\|w\|_X)
    \end{displaymath}
    as $\left\| w \right\|_X \to 0$.
    If such an operator $\ell$ exists, we call it the \emph{specular Fr\'echet differential} of $f$ at $x$ and write $\sFd_x f(x) := \ell$. 
    If there is no confusion, we simply write $\sFd f(x) := \sFd_x f(x)$.

    If  $f$ is specularly G\^ateaux (resp. specularly Fr\'echet) differentiable at every point in $\Omega$, we say that $f$ is \emph{specularly G\^ateaux} (resp. \emph{specularly Fr\'echet}) \emph{differentiable} in $\Omega$.
\end{definition}

The specular G\^ateaux derivative and the specular Fr\'echet differential are unique if they exist; see
\cref{prop:uniqueness_sGd,prop:uniqueness_sFd}.
\cref{fig:differentiability} summarizes the relationships among classical Fr\'echet, classical G\^ateaux, specular Fr\'echet, and specular G\^ateaux differentiability.

\begin{figure}[H]
  \centering
  \begin{tikzcd}[
    row sep=3.5em, column sep=8em,
    nodes={align=center},
    execute at end picture={
      \draw[line width=0.5pt]
        ([xshift=-6pt,yshift=-6pt]current bounding box.south west)
        rectangle
        ([xshift=6pt,yshift=6pt]current bounding box.north east);
    }
  ]
    \begin{tabular}{c} classical Fr\'echet \\ differentiability \end{tabular}
    \arrow[r, Rightarrow, "\text{\small well-known}"{above, yshift=0.5ex}]
    \arrow[d, Rightarrow, "\text{\cref{prop:sF_implies_F}\,}"'] 
    & \begin{tabular}{c} classical G\^ateaux \\ differentiability \end{tabular}
    \arrow[d, Rightarrow, "\text{\small \, \cref{prop:G_implies_spG}}"] \\
    \begin{tabular}{c} specular Fr\'echet \\ differentiability \end{tabular}
    \arrow[r, Rightarrow, "\text{\small \cref{prop:spF_implies_spG}}"{below, yshift=-0.5ex}'] 
    & \begin{tabular}{c} specular G\^ateaux \\ differentiability \end{tabular}
  \end{tikzcd}
  \caption{Relations between classical and specular notions of differentiability.}
  \label{fig:differentiability}
\end{figure}

In this paper, we also consider subdifferentials of semi-convex functionals.
Let $X$ be a normed vector space over $\mathbb{R}$ equipped with a norm $\| \cdot \|_X$.
Let $f : X \to \overline{\mathbb{R}}$ be a functional.
The \emph{effective domain} of $f$ is defined by $\dom (f) := \left\{ x \in X : f(x) < \infty \right\}$.
We say $f$ is \emph{proper} if $\dom (f) \neq \varnothing$ and $f(x) > -\infty$ for all $x \in X$.

For $\lambda \in \mathbb{R}$, we say $f$ is $\lambda$-\emph{convex} if
\begin{displaymath}
    f((1-t)x + tw) \leq (1 - t) f(x) + t f(w) - \frac{\lambda}{2} t (1 - t) \left\| x - w \right\|_X^2
\end{displaymath}
for all $x,w \in \dom(f)$ and all $t \in (0,1)$.
We say that $f$ is \emph{semi-convex} if it is $\lambda$-convex for some $\lambda \in \mathbb{R}$.
This is $\lambda$-convexity along affine segments in the sense of \cite[Def.~2.4.1]{2008_Ambrosio}.
For the classical case $\lambda = 0$, we refer the reader to \cite{1970_Rockafellar_BOOK,2004_Boyd_BOOK} in finite-dimensional spaces and \cite{2002_Zualinescu} in locally convex spaces.
The case $\lambda > 0$ corresponds to strong convexity; see \cite{2018_Nesterov_BOOK} in finite-dimensional spaces and \cite[Sec.~3.5]{2002_Zualinescu} in normed spaces.

Let $x$ be in $\dom (f)$.
If $f$ is proper, the \emph{Fr\'echet subdifferential} of $f$ at $x$ is the set 
\begin{displaymath}
    \widehat{\partial} f(x) := \left\{ \ell \in X^{\ast}: \liminf_{w \to x} \dfrac{f(w)-f(x)-\left\langle \ell, w - x \right\rangle}{\|w - x\|_X} \geq 0\right\}.
\end{displaymath}
If $\widehat{\partial} f(x) \neq \varnothing$, then $f$ is said to be \emph{Fr\'echet subdifferentiable} at $x$.
For Fr\'echet subdifferentials in Banach spaces, we refer the reader to \cite{2003_Kruger,2006_Mordukhovich_BOOK,2006_Mordukhovich,2007_Schirotzek_BOOK}.
If $f$ is proper and convex, the \emph{subdifferential} of $f$ at $x$ is the set 
\begin{displaymath}
    \partial f(x) := \left\{ \ell \in X^{\ast} : f(w) \geq f(x) + \left\langle \ell, w - x \right\rangle \text{ for all } w \in X \right\}.
\end{displaymath}
If $\partial f(x) \neq \varnothing$, then $f$ is said to be \emph{subdifferentiable} at $x$.
Each element of $\partial f(x)$ is called a \emph{subgradient} of $f$ at $x$.
For proper convex functions, the two notions coincide: $\widehat{\partial} f(x) = \partial f(x)$.
On subdifferentials, we refer the reader to \cite{1970_Rockafellar_BOOK,1998_Rockafellar,2018_Nesterov_BOOK} in finite-dimensional spaces, \cite{1990_Clarke_BOOK,2006_Mordukhovich_BOOK} in Banach spaces, and \cite{2002_Zualinescu} in locally convex spaces.
See also \cite{2008_Ambrosio} for a discussion of how the ideas underlying (Fr\'echet) subdifferentiability extend to general metric spaces.

One of the main results of this paper shows that the specular G\^ateaux derivative of a semi-convex function belongs to its Fr\'echet subdifferential; see \cref{thm:sGd_is_subd}.

Let $X = H$ be a Hilbert space over $\mathbb{R}$ equipped with a norm $\| \cdot \|_H$ and an inner product $\left\langle \, \cdot \, , \, \cdot \, \right\rangle _H$, and let $Y = \mathbb{R}$ be equipped with the absolute value norm $\left\vert \, \cdot \, \right\vert$.
In this paper, we do not omit the subscript $H$ on the inner product to distinguish it from the duality pairing.
We generalize the notion of gradient in the sense of specular differentiation by applying the Riesz representation theorem:
\begin{definition}
    Let $H$ be a Hilbert space over $\mathbb{R}$ equipped with a norm $\| \cdot \|_H$ and an inner product $\left\langle \, \cdot \, , \, \cdot \, \right\rangle _H$.
    Assume that a functional $f : \Omega \to \mathbb{R}$ is specularly Fr\'echet differentiable at $x \in \Omega$, where $\Omega$ is an open subset of $H$.
    We define the \emph{specular gradient} of $f$ at $x$ as the unique vector $\sg_x f(x) \in H$ such that 
    \begin{equation}    \label{eq: specular gradient}
        \left\langle \, \sFd f(x), w \right\rangle = \left\langle \sg_x f(x), w \right\rangle_H
    \end{equation}
    for all $w \in H$.
    If there is no confusion, we simply write $\sg f(x) := \sg_x f(x)$.
\end{definition}

\subsection{Organization}

The remainder of this paper is organized as follows.
\Cref{sec:spdiff} presents the geometric motivation for specular directional derivatives and derives the corresponding formula.
\Cref{sec:sdd,sec:sGd,sec:sFd} then study specular directional derivatives, specular G\^ateaux derivatives, and specular Fr\'echet differentials, respectively.
These concepts are presented under progressively stronger differentiability assumptions, with each section addressing both vector-valued and real-valued cases.
In \cref{sec:sGd}, we establish the Quasi-Mean Value Theorem for real-valued functions and derive a quadratic lower bound and the corresponding Fr\'echet subdifferential inclusion for semi-convex functionals.
In \cref{sec:sFd}, we characterize specular Fr\'echet differentiability by uniform convergence of normalized specular difference quotients on the unit sphere.
We also prove its equivalence with specular G\^ateaux differentiability for locally Lipschitz mappings defined on open subsets of $\mathbb{R}^n$ and describe specular gradients in Hilbert spaces.

\section{Specular differentiation} \label{sec:spdiff}

Here, we explain what motivates the formula \eqref{eq:sdd}.
Observe that \eqref{eq:sdd} can be rewritten as 
\begin{multline} \label{eq:sdd_UV}
    \partial_{v}^{\sd}f(x) 
    =
    \lim_{h \searrow 0} \left[ \left( \dfrac{f(x + hv) - f(x)}{h} \right) \dfrac{\left\| U \right\|_{X \times Y}}{\left\| U \right\|_{X \times Y} + \left\| V \right\|_{X \times Y}} \right. \\ 
    + \left. \left( \dfrac{f(x) - f(x - hv)}{h} \right) \dfrac{\left\| V \right\|_{X \times Y}}{\left\| U \right\|_{X \times Y} + \left\| V \right\|_{X \times Y}} \right],
\end{multline}
where 
\begin{equation}    \label{def:pt_UV}
    U := (hv, f(x) - f(x - hv)) 
    \qquad\text{and}\qquad        
    V := (hv, f(x + hv) - f(x))
\end{equation}
for sufficiently small $h > 0$.  

Consider $X = \mathbb{R}^n$ and $Y = \mathbb{R}$.
Let $f : \Omega \to \mathbb{R}$ be a function, where $\Omega$ is an open subset of $\mathbb{R}^n$.
Fix $x \in \Omega$ and $v \in \mathbb{R}^n \setminus \{0\}$.
Based on \cref{def: specular Gateaux derivative}, we further denote $A := (x - hv, f(x))$, $C := (x, f(x))$, $E := (x + hv, f(x))$, $L := (x - hv, f(x - hv))$, $R := (x + hv, f(x + hv))$ for $h > 0$.
\Cref{fig:motivation of the definition of specular directional derivatives} illustrates the situation.
Note that 
\begin{displaymath}
    \left\vert \overline{AC} \right\vert
    = \left\| A - C \right\|_{\mathbb{R}^n \times \mathbb{R}} 
    = h \| v \|_{\mathbb{R}^n}
    = \left\| C - E \right\|_{\mathbb{R}^n \times \mathbb{R}}
    = \left\vert \overline{CE} \right\vert.
\end{displaymath}
Let $r > 0$ be such that $0 < r < h \| v \|_{\mathbb{R}^n}$.
Let $P$ and $Q$ be the intersection points in $\mathbb{R}^n \times \mathbb{R}$ of the line segments $\overline{LC}$ and $\overline{RC}$ with the sphere $\partial B(C, r)$ centered at $C$ with radius $r$, i.e.,
\begin{displaymath}
    \partial B(C, r) = \left\{ (y, s) \in \mathbb{R}^n \times \mathbb{R} : \left\| C - (y, s) \right\|_{\mathbb{R}^n \times \mathbb{R}} = r\right\}.
\end{displaymath}
Let $B$ and $D$ be the feet of the perpendiculars from $P$ and $Q$ to the line $\overline{AE}$, respectively.
For the following geometric argument, assume that $f(x-hv) \neq f(x)$ and $f(x+hv) \neq f(x)$.
Then, triangles $\triangle CAL$, $\triangle CBP$, $\triangle CDQ$, and $\triangle CER$ are right triangles with right angles at $A$, $B$, $D$, and $E$, respectively.
Let $F$ and $G$ be the intersection points of the sphere $\partial B(C, r)$ and the line segment that passes through $C$ and is parallel to the line $\overline{PQ}$.

The term ``specular'' comes from the behavior of light, which reflects from a smooth surface at the same angle at which it arrives.
The lines $\overline{LC}$ and $\overline{CR}$ represent the incident ray and the reflected ray, respectively, and the line $\overline{FG}$ represents a mirror.
Then the angles the rays make with the mirror are equal, i.e., $\angle LCF = \angle RCG$.

\begin{figure}[htbp]
    \centering
    \includegraphics[width=0.9\textwidth]{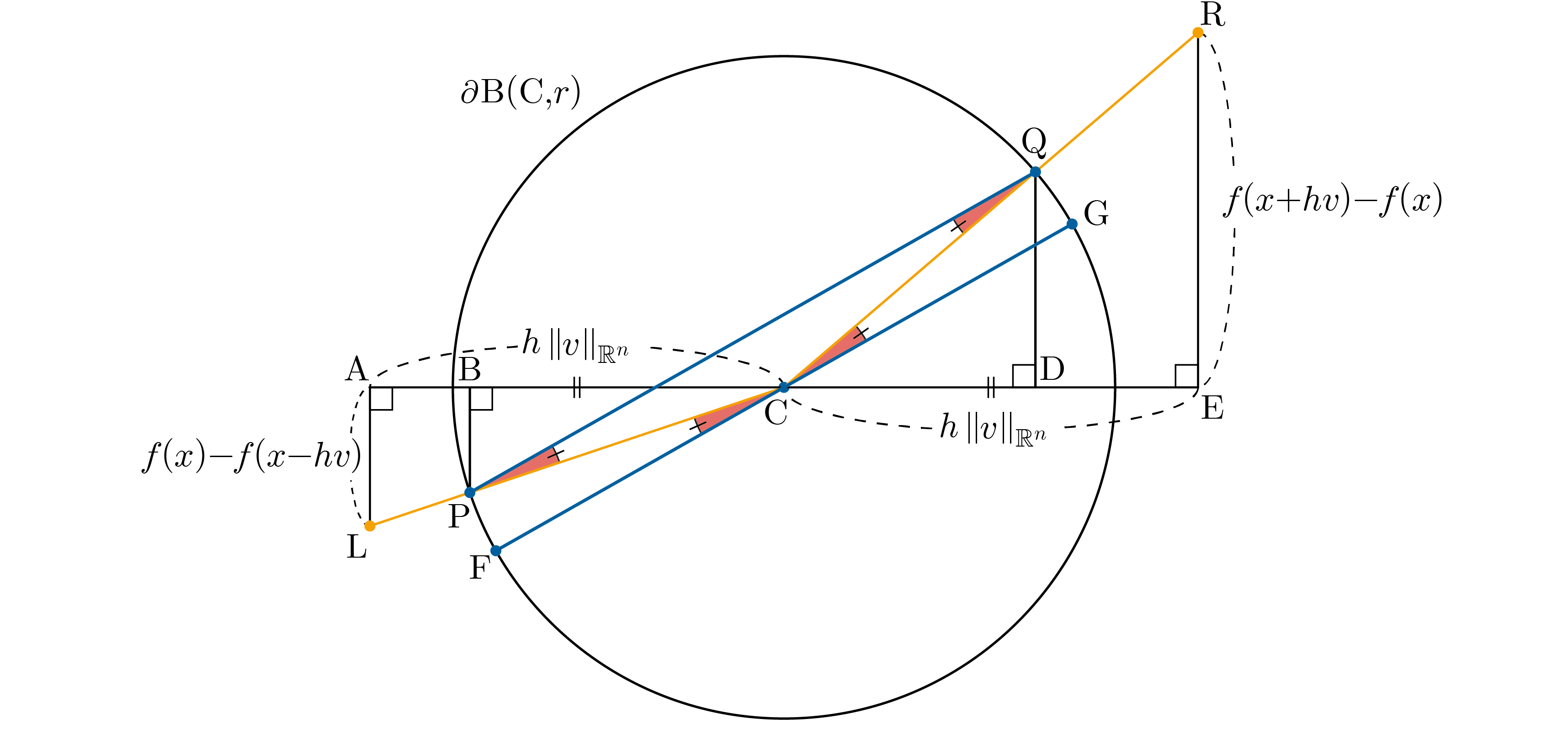}
    \vspace{-1.0\baselineskip}
    \caption{Motivation of the definition of specular directional derivatives.}
    \label{fig:motivation of the definition of specular directional derivatives}
\end{figure}

We want to find the slope of the line segment $\overline{PQ}$, which converges to $\|v\|_{\mathbb{R}^n}^{-1} \partial^{\sd}_v f(x)$ as $h \searrow 0$, whenever this derivative exists.
Note that the angles $\angle LCF$, $\angle RCG$, $\angle CQP$, and $\angle CPQ$ are equal.
Since $\triangle CDQ$ and $\triangle CER$ are similar, it holds that
\begin{displaymath}
    \dfrac{\left\| C - R \right\|_{\mathbb{R}^n \times \mathbb{R}}}{r} 
    = \dfrac{\left\vert \overline{CR} \right\vert }{\left\vert \overline{CQ} \right\vert } 
    = \dfrac{\left\vert \overline{CE} \right\vert }{\left\vert \overline{CD} \right\vert } 
    = \dfrac{h \| v \|_{\mathbb{R}^n}}{\left\vert \overline{CD} \right\vert}
\end{displaymath}
and 
\begin{displaymath}
    \dfrac{\left\| C - R \right\|_{\mathbb{R}^n \times \mathbb{R}}}{r} 
    = \dfrac{\left\vert \overline{CR} \right\vert }{\left\vert \overline{CQ} \right\vert } 
    = \dfrac{\left\vert \overline{ER} \right\vert }{\left\vert \overline{DQ} \right\vert } 
    = \dfrac{|f(x + hv) - f(x)|}{\left\vert \overline{DQ} \right\vert},
\end{displaymath}
which implies that 
\begin{displaymath}
    \left\vert \overline{CD} \right\vert = \dfrac{r h \| v \|_{\mathbb{R}^n}}{\left\| C - R \right\|_{\mathbb{R}^n \times \mathbb{R}}}  
    \qquad\text{and}\qquad
    \left\vert \overline{DQ} \right\vert = \dfrac{r \, |f(x + hv) - f(x)|}{\left\| C - R \right\|_{\mathbb{R}^n \times \mathbb{R}}}.
\end{displaymath}
Similarly, using the fact that $\triangle CBP$ and $\triangle CAL$ are similar, one can find that 
\begin{displaymath}
    \left\vert \overline{BC} \right\vert = \dfrac{r h \| v \|_{\mathbb{R}^n}}{\left\| C - L \right\|_{\mathbb{R}^n \times \mathbb{R}}}  
    \qquad\text{and}\qquad
    \left\vert \overline{BP} \right\vert = \dfrac{r \, |f(x) - f(x - hv)|}{\left\| C - L \right\|_{\mathbb{R}^n \times \mathbb{R}}}.
\end{displaymath}
Interpreting $\overline{BP}$ and $\overline{DQ}$ as signed vertical displacements, positive upwards, the slope of $\overline{PQ}$ is
\begin{multline*}
  \dfrac{\overline{DQ} - \overline{BP}}{\left\vert \overline{CD} \right\vert + \left\vert \overline{BC} \right\vert}
    = \left( \dfrac{f(x + hv) - f(x)}{h \| v \|_{\mathbb{R}^n}} \right)\dfrac{\left\| C - L \right\|_{\mathbb{R}^n \times \mathbb{R}}}{\left\| C - L \right\|_{\mathbb{R}^n \times \mathbb{R}} + \left\| C - R \right\|_{\mathbb{R}^n \times \mathbb{R}}}\\
    + \left( \dfrac{f(x) - f(x - hv)}{h \| v \|_{\mathbb{R}^n}} \right)\dfrac{\left\| C - R \right\|_{\mathbb{R}^n \times \mathbb{R}}}{\left\| C - L \right\|_{\mathbb{R}^n \times \mathbb{R}} + \left\| C - R \right\|_{\mathbb{R}^n \times \mathbb{R}}},
\end{multline*}
which does not depend on the choice of $r$.
By ignoring the factor $\| v \|_{\mathbb{R}^n}^{-1}$, as is done in the definition of classical derivatives, we obtain the fraction in \eqref{eq:sdd_UV} since $U = C - L$ and $V = R - C$.

\section{Specular directional derivatives}  \label{sec:sdd}
 
In this section, we study the relationships between specular and classical directional derivatives.
We also derive representations of specular directional derivatives of real-valued functions.
Under suitable assumptions, specular directional derivatives can be expressed in terms of one-sided directional derivatives, yielding three corollaries.

\subsection{Vector-valued functions}

Specular directional derivatives generalize classical directional derivatives.

\begin{proposition}  \label{prop:dd_implies_sdd}
    Let $X$ and $Y$ be normed vector spaces over $\mathbb{R}$ equipped with norms $\| \cdot \|_X$ and $\| \cdot \|_Y$, respectively.
    Let $f : \Omega \to Y$ be a function, where $\Omega$ is an open subset of $X$.
    Let $x \in \Omega$ and $v \in X$ be fixed.
    If the directional derivative $\partial_v f(x)$ exists, then the specular directional derivative $\partial^{\sd}_v f(x)$ exists with $\partial^{\sd}_v f(x) = \partial_v f(x)$.
\end{proposition}

\begin{proof}
    If $v = 0$, then $\partial_v f(x) = 0$ and $\partial^{\sd}_v f(x) = 0$.
    Thus, assume that $v\neq 0$.
    Let $\varepsilon > 0$ be arbitrary.
    For sufficiently small $h > 0$, write 
    \begin{displaymath}
        a := \left\| U \right\|_{X \times Y}, \qquad
        b := \left\| V \right\|_{X \times Y}, \qquad
        c := \partial_v f(x) \in Y,
    \end{displaymath}
    where $U$ and $V$ are defined as in \eqref{def:pt_UV}.
    Then, $a \geq h \| v \|_X > 0$ and $b \geq h \| v \|_X > 0$.
    From the existence of $\partial_v f(x)$, there exists $\delta > 0$ such that, if $0 < h < \delta$, then 
    \begin{displaymath}
        \left\| \dfrac{f(x + hv) - f(x)}{h} - c \, \right\|_Y < \varepsilon 
        \qquad\text{and}\qquad
        \left\| \dfrac{f(x) - f(x - hv)}{h} - c \, \right\|_Y < \varepsilon.
    \end{displaymath}
    If $0 < h < \delta$, then 
    \begin{align*}
        & \left\| \dfrac{a (f(x + hv) - f(x))  + b (f(x) - f(x - hv))}{h(a + b)} - c \, \right\|_Y \\
        =& \left\| \dfrac{a}{a + b} \left( \dfrac{f(x + hv) - f(x)}{h} - c \right) + \dfrac{b}{a + b} \left( \dfrac{f(x) - f(x - hv)}{h} - c \right) \right\|_Y \\ 
        \leq& \,  \dfrac{a}{a + b} \left\|  \dfrac{f(x + hv) - f(x)}{h} - c  \, \right\|_Y + \dfrac{b}{a + b} \left\| \dfrac{f(x) - f(x - hv)}{h} - c \, \right\|_Y  \\
        <& \, \varepsilon.
    \end{align*}
    Since $\varepsilon > 0$ is arbitrary, we conclude that $\partial^{\sd}_v f(x)$ exists and $\partial^{\sd}_v f(x) = c = \partial_v f(x)$.    
\end{proof}

In the classical sense, the directional derivative of a function $f : X \to Y$ can be rewritten as 
\begin{equation}    \label{eq: representation of classical directional derivatives}
    \left. \dfrac{d}{dt} f(x + tv) \right|_{t = 0} = \partial_v f(x)
\end{equation}
for $x, v \in X$.
However, such a representation may not hold for specular directional derivatives because of the nonlinearity.
The following example relates specular directional derivatives to specular derivatives along affine lines.

\begin{example}    \label{ex:sdd_along_line}
    Let $X$ and $Y$ be normed vector spaces over $\mathbb{R}$ equipped with norms $\| \cdot \|_X$ and $\| \cdot \|_Y$, respectively.
    Let $f : \Omega \to Y$ be a function, where $\Omega$ is an open subset of $X$.
    Let $x$ be in $\Omega$, and let $v \in X \setminus \left\{ 0 \right\}$ be fixed. 
    Choose $\varepsilon > 0$ such that $x+tv \in \Omega$ for all $t \in (-\varepsilon, \varepsilon)$.
    Introduce a function $F : (-\varepsilon,\varepsilon) \to Y$ defined by
    \begin{displaymath}
        F(t) := \dfrac{f(x + tv)}{\| v \|_X}.
    \end{displaymath}
    Suppose that $\partial^{\sd}_v f(x)$ and $F^{\sd}(0)$ exist.
    Applying \eqref{eq:A_subscript}, we find that 
    \begin{align*}
        \frac{\partial_v^{\sd}f(x)}{\|v\|_X}
        &=
        \lim_{t \searrow 0} \frac{1}{\|v\|_X} \mathcal{A}_{\|v\|_X}^Y \bigl( f(x+tv) - f(x), f(x) - f(x-tv), t \bigr)\\
        &=
        \lim_{t \searrow 0} \mathcal{A}^Y \bigl( F(t) - F(0), F(0) - F(-t), t \bigr)\\
        &=
        F^{\sd}(0),
    \end{align*}
    that is,
    \begin{equation}    \label{eq: representation of specular directional derivatives - 1}
        \left. \dfrac{\sGd}{\sGd t} \dfrac{f(x + tv)}{\| v \|_X} \right|_{t = 0} = \dfrac{\partial^{\sd}_v f(x)}{\| v \|_X}.
    \end{equation}

    Similarly, letting $G(t) := f(x + tv)$ for $t \in (-\varepsilon,\varepsilon)$, one can find that 
    \begin{equation}    \label{eq: representation of specular directional derivatives - 2}
        \left. \dfrac{\sGd}{\sGd t} f(x + tv) \right|_{t = 0}
        =
        \dfrac{\partial_v^{\sd}\left( \| v \|_X f \right) (x)}{\| v \|_X}
    \end{equation}
    if $\partial_v^{\sd}\left( \| v \|_X f \right) (x)$ and $G^{\sd}(0)$ exist.    

    Observe that if $f$ is smooth enough, then the equalities in \eqref{eq: representation of specular directional derivatives - 1} and \eqref{eq: representation of specular directional derivatives - 2} can be simplified to the equality \eqref{eq: representation of classical directional derivatives}, as the $\| v \|_X$ term cancels out. 
\end{example}

\subsection{Specular tangent vectors}\label{subsec:specular_tangent_vectors}

Regular specular differentiability allows us to introduce tangent vectors to the graph of a function.

\begin{definition}  \label{def:tangent_vector}
    Let $X$ and $Y$ be real normed vector spaces, and let $\Omega$ be an open subset of $X$.
    Let $x\in\Omega$ and $v\in X\setminus\{0\}$.
    Suppose that a function $f:\Omega\to Y$ is regularly specularly differentiable at $x$ in the direction $v$.
    We define the \emph{outward one-sided unit tangent vectors} by
    \begin{displaymath}
        \tau^+(x)
        :=
        \frac{\left( v, \partial_v^+ f(x) \right)}{\left\| \left( v, \partial_v^+ f(x) \right)  \right\|_{X \times Y}}
        \qquad\text{and}\qquad
        \tau^-(x)
        :=
        \frac{\left( -v, -\partial_v^- f(x) \right)}{\left\| \left( v, \partial_v^- f(x) \right) \right\|_{X \times Y} },
    \end{displaymath}
    and the \emph{specular unit tangent vector} by 
    \begin{displaymath}
        \tau^{\sd}(x) 
        :=
        \frac{\left( v, \partial^{\sd}_v f(x) \right)}{\left\| \left( v, \partial^{\sd}_v f(x) \right) \right\|_{X \times Y}}.
    \end{displaymath}
\end{definition}

The specular unit tangent vector is the normalized difference of the outward one-sided unit tangent vectors.

\begin{theorem}\label{thm:specular_tangent_bisector}
    Let $X$ and $Y$ be real normed vector spaces, and let $\Omega$ be an open subset of $X$.
    Let $x\in\Omega$ and $v\in X\setminus\{0\}$.
    Suppose that a function $f:\Omega\to Y$ is regularly specularly differentiable at $x$ in the direction $v$.
    Then $\partial_v^{\sd}f(x)$ exists in $Y$, $\tau^+(x)-\tau^-(x)\neq0$, and
    \begin{equation}\label{eq:specular_tangent_bisector}
        \tau^{\sd}(x)
        =\frac{\tau^+(x)-\tau^-(x)}{\|\tau^+(x)-\tau^-(x)\|_{X\times Y}}.
    \end{equation}
\end{theorem}

\begin{proof}
    Let $\alpha := \partial_v^+ f(x)$, $\beta := \partial_v^- f(x)$, $a:=\|(v, \alpha)\|_{X \times Y}$, $b:=\|(v, \beta)\|_{X \times Y}$.
    Since $v \neq 0$, both $a$ and $b$ are positive.
    Continuity of the norm yields
    \begin{displaymath}
        \frac{\|U\|_{X\times Y}}{h}\to b
        \qquad\text{and}\qquad
        \frac{\|V\|_{X\times Y}}{h}\to a
    \end{displaymath}
    as $h \searrow 0$, where the vectors $U$ and $V$ are defined in \eqref{def:pt_UV}.
    Passing to the limit in \eqref{eq:sdd_UV}, we obtain
    \begin{displaymath}
        \partial_v^{\sd}f(x)=\frac{b \alpha + a \beta}{a + b}\in Y.
    \end{displaymath}
    By \cref{def:tangent_vector},
    \begin{displaymath}
        \tau^+(x) - \tau^-(x)
        =\frac{(v, \alpha)}{a} + \frac{(v, \beta)}{b}
        =\left(\frac{1}{a} + \frac{1}{b}\right) \left(v,\partial_v^{\sd} f(x)\right).
    \end{displaymath}
    The scalar factor is positive, and the vector on the right is nonzero since $v \neq 0$.
    Normalizing both sides proves \eqref{eq:specular_tangent_bisector}.
\end{proof}

Suppose that $X$ and $Y$ are inner product spaces over $\mathbb{R}$ equipped with inner products $\langle \,\cdot\,,\,\cdot\, \rangle_X$ and $\langle \, \cdot \,, \,\cdot \,\rangle_Y$, respectively.
Assume that $\| \cdot  \|_X$ and $\| \cdot \|_Y$ are the corresponding induced norms.
We equip $X \times Y$ with the product inner product
\begin{equation}\label{eq:product_inner_product}
    \bigl\langle(x_1,y_1),(x_2,y_2)\bigr\rangle_{X\times Y}
    :=
    \langle x_1,x_2\rangle_X+\langle y_1,y_2\rangle_Y
\end{equation}
for $x_1,x_2\in X$ and $y_1,y_2\in Y$.
The norm induced by \eqref{eq:product_inner_product} is precisely the product norm in \eqref{def:prod_norm}.

\begin{corollary}\label{cor:specular_tangent_orthogonality}
    Under the assumptions of \cref{thm:specular_tangent_bisector}, suppose additionally that $X$ and $Y$ are inner product spaces over $\mathbb{R}$ equipped with inner products $\langle \, \cdot \, , \, \cdot \, \rangle_X$ and $\langle \, \cdot \, , \, \cdot \, \rangle_Y$, respectively.
    Then the following statements hold:
    \begin{enumerate}[label=\upshape(\alph*)]
        \item $\tau^+(x)+\tau^-(x)$ and $\tau^{\sd}(x)$ are orthogonal:
        \begin{equation}\label{cor:specular_tangent_orthogonality-1}
            \bigl\langle\tau^+(x)+\tau^-(x),\tau^{\sd}(x)\bigr\rangle_{X\times Y}=0.
        \end{equation}
        \item $\tau^+(x)-\tau^-(x)$ and $\tau^{\sd}(x)$ are parallel:
        \begin{equation}\label{cor:specular_tangent_orthogonality-2}
            \frac{\tau^+(x)-\tau^-(x)}{\|\tau^+(x)-\tau^-(x)\|_{X\times Y}}
            =
            \tau^{\sd}(x).
        \end{equation}
        \item \label{cor:specular_tangent_orthogonality-3}
        $\tau^{\sd}(x)$ forms equal angles with $\tau^+(x)$ and $-\tau^-(x)$.
    \end{enumerate}
\end{corollary}

\begin{proof}
    For brevity, write $\tau^+ := \tau^+(x)$, $\tau^- := \tau^-(x)$, and $\tau^{\sd} := \tau^{\sd}(x)$.
    By the symmetry of the inner product,
    \begin{align*}
        \bigl\langle \tau^+ + \tau^-, \tau^+ - \tau^- \bigr\rangle_{X\times Y}
        &= \|\tau^+\|_{X\times Y}^2 - \langle\tau^+, \tau^- \rangle_{X\times Y} + \langle\tau^-, \tau^+\rangle_{X\times Y} - \|\tau^-\|_{X\times Y}^2   \\
        &= \|\tau^+\|_{X\times Y}^2 -\|\tau^-\|_{X\times Y}^2  \\
        &= 0,
    \end{align*}
    where the last equality follows since both $\tau^+$ and $\tau^-$ are unit vectors.
    By \cref{thm:specular_tangent_bisector}, $\tau^+ - \tau^- \neq 0$, and hence
    \begin{displaymath}
        \bigl\langle\tau^+ + \tau^-, \tau^{\sd} \bigr\rangle_{X \times Y}
        =
        \frac{\bigl\langle\tau^+ + \tau^-, \tau^+ - \tau^- \bigr\rangle_{X\times Y}}{\|\tau^+ - \tau^- \|_{X\times Y}}
        =
        0.
    \end{displaymath}
    Therefore, \eqref{cor:specular_tangent_orthogonality-1} follows.

    Next, applying \cref{thm:specular_tangent_bisector} yields
    \begin{displaymath}
        \left\langle \frac{\tau^+-\tau^-}{\|\tau^+-\tau^-\|_{X\times Y}}, \tau^{\sd} \right\rangle_{X\times Y}
        =
        \frac{\langle\tau^+-\tau^-,\tau^+-\tau^-\rangle_{X\times Y}}{\|\tau^+-\tau^-\|_{X\times Y}^2}
        =
        1.        
    \end{displaymath}
    Since the two vectors in the inner product on the left-hand side are unit vectors, the angle between them is zero.
    Hence, $\tau^+ - \tau^-$ and $\tau^{\sd}$ are parallel, proving \eqref{cor:specular_tangent_orthogonality-2}.

    We prove \cref{cor:specular_tangent_orthogonality-3}.
    The orthogonality implies
    \begin{displaymath}
        \langle\tau^+, \tau^{\sd} \rangle_{X\times Y}
        =
        -\langle\tau^-, \tau^{\sd} \rangle_{X\times Y}
        =
        \langle-\tau^-, \tau^{\sd} \rangle_{X\times Y}.
    \end{displaymath}
    Since $\tau^+$, $-\tau^-$, and $\tau^{\sd}$ are unit vectors, these inner products are the cosines of the corresponding angles.
    Thus, $\tau^{\sd}$ forms equal angles with $\tau^+$ and $-\tau^-$.
\end{proof}

\begin{figure}[h]
    \centering
    \includegraphics[width=1\textwidth]{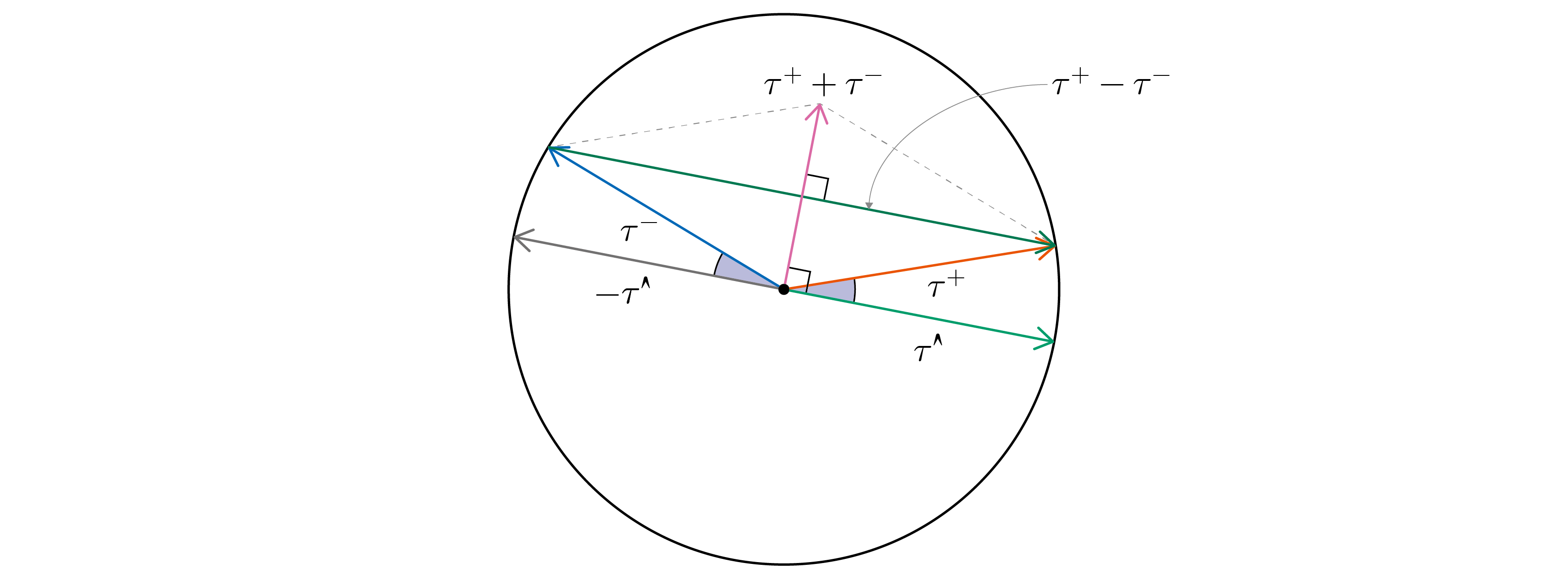}
    \vspace{-1.0\baselineskip}
    \caption{Geometric interpretation of the specular unit tangent vector in inner product spaces.}
    \label{fig:tangent_vectors}
\end{figure}

\begin{example}\label{ex:relu_specular_tangent}
    Let $f:\mathbb{R}\to\mathbb{R}$ be defined by $f(x) := \max\{0, x\}$.
    At $x = 0$ in the direction $v = 1$, we have 
    \begin{displaymath}
        \left( \tau^+(0) + \tau^-(0) \right) \innerprd \tau^{\sd}(0)
        =
        \left(\frac{(1, 1)}{\sqrt{2}} + (-1, 0)\right) \innerprd \frac{\left(1,\sqrt{2}-1\right)}{\sqrt{4 - 2 \sqrt{2}}}
        =
        0,
    \end{displaymath}
    and hence \eqref{cor:specular_tangent_orthogonality-1} holds in this example.
\end{example}

\begin{remark}\label{rem:graph_curve_not_direct_tangent}
    Instead of \cref{def:tangent_vector}, one could define tangent vectors by applying \cref{ex:sdd_along_line} to parametrized graph curves.
    However, this may lead to a less direct geometric interpretation.

    For example, let $f$ be as in \cref{ex:relu_specular_tangent} and define the graph map $\gamma:\mathbb{R}\to\mathbb{R}^2$ by $\gamma(t) := (t,f(t))$.
    Applying
    \eqref{eq: representation of specular directional derivatives - 2}
    to $\gamma$ at $x=0$ in the direction $v=1$, we obtain
    \begin{displaymath}
        \left.\frac{\sGd}{\sGd t}\gamma(t)\right|_{t=0}
        = 
        \gamma^{\sd}(0)
        = 
        \lim_{h\searrow 0} \frac{(h, h)\sqrt{2}h + (h, 0)\sqrt{3}h}{h(\sqrt{2}h + \sqrt{3}h)}
        =
        \left(1,\sqrt{6}-2\right).
    \end{displaymath}
    After normalization, this yields
    \begin{displaymath}
        \frac{\gamma^{\sd}(0)}
        {\|\gamma^{\sd}(0)\|_{\mathbb{R}^2}}
        =
        \frac{\left( 1,\sqrt{6}-2 \right)}{\sqrt{11-4\sqrt{6}}}
        \neq
        \frac{\left( 1,\sqrt{2}-1 \right)}{\sqrt{4-2\sqrt{2}}}
        =
        \tau^{\sd}(0).
    \end{displaymath}
    The additional parameter coordinate may distort the tangent direction relative to the original graph geometry.
    Consequently, the resulting unit vector need not satisfy the orthogonality relation \eqref{cor:specular_tangent_orthogonality-1}.
\end{remark}

\subsection{Real-valued functions}

In this section, we concentrate on the case $Y = \mathbb{R}$ equipped with the absolute value norm $\| \cdot \|_Y = \left\vert \, \cdot \, \right\vert$.
We simply write $\| \cdot \|$ for the norm $\| \cdot \|_X$ if there is no confusion.

Auxiliary functions were introduced in \cite{2026a_Jung,2026b_Jung} to simplify expressions arising in specular differentiation of real-valued functions.
Here, we use scaled versions of these functions with a parameter $\sigma > 0$.

\begin{definition}
    For $\sigma > 0$, define the function $\mathcal{B}_{\sigma} : \overline{\mathbb{R}} \times \overline{\mathbb{R}} \to \overline{\mathbb{R}}$ by
    \begin{equation}    \label{def:B}
        \mathcal{B}_{\sigma}(\alpha, \beta) := \sigma \tan \left( \frac{1}{2} \arctan \frac{\alpha}{\sigma} + \frac{1}{2} \arctan \frac{\beta}{\sigma} \right),
    \end{equation}
    and the function $\mathcal{C}_{\sigma} : \overline{\mathbb{R}} \times \overline{\mathbb{R}} \to \overline{\mathbb{R}}$ by
    \begin{equation} \label{def:C}
    \mathcal{C}_{\sigma}(\alpha, \beta) :=
    \begin{cases}
        \dfrac{\alpha \left\| (\sigma, \alpha) \right\|_{\mathbb{R}^2}^{-1} + \beta \left\| (\sigma, \beta) \right\|_{\mathbb{R}^2}^{-1}}{\left\| (\sigma, \alpha) \right\|_{\mathbb{R}^2}^{-1} + \left\| (\sigma, \beta) \right\|_{\mathbb{R}^2}^{-1}}
        & \text{if $\alpha, \beta \in \mathbb{R}$,}  \\[1em]
        \alpha \pm \left\| (\sigma, \alpha) \right\|_{\mathbb{R}^2}
        & \text{if $\alpha \in \mathbb{R}$, $\beta = \pm \infty$,}   \\
        \beta \pm \left\| (\sigma, \beta) \right\|_{\mathbb{R}^2}
        & \text{if $\alpha = \pm \infty$, $\beta \in \mathbb{R}$,}  \\
        0 
        & \text{if $\alpha = \pm \infty, \beta = \mp \infty$,} \\
        \pm \infty 
        & \text{if $\alpha = \beta = \pm \infty$}  .
    \end{cases}
    \end{equation}
    Here, we use the conventions $\arctan(\pm\infty) := \pm\frac{\pi}{2}$ and $\tan(\pm\frac{\pi}{2}):=\pm\infty$.
    When $\sigma = 1$, we simply write $\mathcal{B} := \mathcal{B}_1$ and $\mathcal{C} := \mathcal{C}_1$.
\end{definition}

The parameters $\alpha$ and $\beta$ are intended to represent one-sided directional derivatives.
See \cite[App.~A]{2026a_Jung} for analysis of the functions $\mathcal{A}$, $\mathcal{B}$, and $\mathcal{C}$.
The values of $\mathcal{C}$ for extended real arguments are motivated by \cite[Lem.~A.2 (m)]{2026a_Jung}.
See \cite[App.~A]{2026b_Jung} for an analysis of $\mathcal{C}_{\sigma}$ on $\mathbb{R}^2$.
These functions are particularly useful in formulating specular G\^ateaux derivatives, specular Fr\'echet differentials, and their estimates; see \cref{cor:repr_sdd_C} and \cref{thm:sFd_equivalence}.
Note that, for each $(a, b, c) \in \mathbb{R} \times \mathbb{R} \times (0, \infty)$, it holds that 
\begin{equation}    \label{eq:ABC}
    \mathcal{A}(a, b, c)
    = \mathcal{B}\left( \frac{a}{c}, \frac{b}{c} \right)
    = \mathcal{C}\left( \frac{a}{c}, \frac{b}{c} \right)
\end{equation}
by \cite[Lem.~2.1]{2026a_Jung}.
Also, observe that, for each $(\alpha, \beta) \in \overline{\mathbb{R}} \times \overline{\mathbb{R}}$, it holds that 
\begin{equation}    \label{eq:BC}
    \mathcal{B}(\alpha, \beta) = \mathcal{C}(\alpha, \beta)
\end{equation}
by \cite[Lem.~2.1]{2026a_Jung}.

Specular directional derivatives of real-valued functions can be represented in terms of the tangent and arctangent functions as follows.

\begin{lemma}   \label{lem:rep_of_sdd}
    Let $X$ be a normed vector space over $\mathbb{R}$ equipped with a norm $\| \cdot \|$.
    Let $f : \Omega \to \mathbb{R}$ be a function, where $\Omega$ is an open subset of $X$.
    For each $x \in \Omega$ and $v \in X \setminus \{0\}$, the following identities hold:
    \begin{align}
        \partial_{v}^{\sd}f(x) 
        &= \lim_{h \searrow 0} \mathcal{A}_{\|v\|} \left( f(x+hv) - f(x), f(x) - f(x-hv), h \right)    \nonumber  \\
        &= \lim_{h \searrow 0} \mathcal{B}_{\|v\|} \left( \dfrac{f(x+hv) - f(x)}{h}, \dfrac{f(x) - f(x-hv)}{h} \right)    \label{lem:rep_of_sdd-2}   \\
        &= \lim_{h \searrow 0} \mathcal{C}_{\|v\|} \left( \dfrac{f(x+hv) - f(x)}{h}, \dfrac{f(x) - f(x-hv)}{h} \right),  \nonumber
    \end{align}
    whenever one of the limits exists as a real number.
\end{lemma}

\begin{proof}
    For sufficiently small $h > 0$, write 
    \begin{displaymath}
        a := f(x + hv) - f(x), \qquad
        b := f(x) - f(x - hv), \qquad
        r := \| v \|.
    \end{displaymath}
    By \eqref{eq:A_subscript}, \eqref{eq:ABC}, and the definitions of $\mathcal{B}_{\sigma}$ and $\mathcal{C}_{\sigma}$, we obtain
    \begin{displaymath}
        \mathcal{A}_r(a,b,h)
        = r\mathcal{A}(a,b,rh)
        = r\mathcal{C} \left(\frac{a}{rh},\frac{b}{rh}\right)
        = \mathcal{C}_r \left(\frac{a}{h},\frac{b}{h}\right)
        =\mathcal{B}_r\left(\frac{a}{h},\frac{b}{h}\right).
    \end{displaymath}
    Taking the limit as $h \searrow 0$ yields the desired identities by \eqref{eq:sdd}.
\end{proof}

\begin{remark}
    The one-sided limit $(h \searrow 0)$ in \eqref{lem:rep_of_sdd-2} exists if and only if the limit $(h \to 0)$ exists.
\end{remark}

\begin{remark}
    The specular directional derivative may not be invariant under the choice of the product norm in \eqref{def:prod_norm}.
    For example, consider the function $f : \mathbb{R}^2 \to \mathbb{R}$ defined by $f(x_1, x_2) = x_1 + |x_2|$ for $(x_1, x_2) \in \mathbb{R}^2$.
    Fix the direction $v = (1, 1)$.
    On the one hand, one can calculate that 
    \begin{equation} \label{rmk:not_invariant}
        \partial^{\sd}_v f(0, 0) = \sqrt{3} - 1
    \end{equation}
    by applying \cref{lem:rep_of_sdd}.
    On the other hand, consider the alternative product norm 
    \begin{displaymath}
        \left\| ((x_1, x_2), y) \right\|_{2 \times 1} := \sqrt{x_1^2 + x_2^2} + |y|  
    \end{displaymath}
    for $((x_1, x_2), y) \in \mathbb{R}^2 \times \mathbb{R}$, instead of \eqref{def:prod_norm}.
    Replacing the product norm in the limit \eqref{eq:sdd_UV} with this norm, we have
    \begin{displaymath}
       \lim_{h \searrow 0} \left[ \left( \dfrac{f(h, h)}{h} \right) \dfrac{\left\| U\right\|_{2 \times 1}}{\left\| U \right\|_{2 \times 1} + \left\| V \right\|_{2 \times 1}} - \left( \dfrac{f(-h, -h)}{h} \right) \dfrac{\left\| V \right\|_{2 \times 1}}{\left\| U \right\|_{2 \times 1} + \left\| V \right\|_{2 \times 1}} \right]
       = 2 - \sqrt{2},
    \end{displaymath}
    where $U = ((-h, -h), f(-h, -h))$ and $V = ((h, h), f(h, h))$, which differs from the value in \eqref{rmk:not_invariant}.
\end{remark}

For a function $f : \Omega \to \mathbb{R}$ and vectors $x \in \Omega$, $v \in X$, where $\Omega$ is an open subset of $X$, we consider the following hypotheses:
\begin{enumerate}[label=\rm{(H\arabic*)}, ref=\rm{(H\arabic*)}, itemsep=1ex]
    \item $\partial^+_v f(x)$ exists as an extended real number. \label{H1}
    \item $\partial^-_v f(x)$ exists as an extended real number. \label{H2}
    \item $\partial_v^{+} f(x)$ and $\partial^-_v f(x)$ are not simultaneously $\infty$ or $-\infty$. \label{H3}
\end{enumerate}

Recall that specular differentiability does not guarantee the existence of both one-sided directional derivatives; see \cite[Ex.~2.3]{2026a_Jung}.
Therefore, specular directional differentiability does not imply either \ref{H1} or \ref{H2} in general.
However, if a specular directional derivative exists and one of the one-sided directional derivatives exists, then both one-sided directional derivatives exist and the specular directional derivative admits a representation in terms of them.
This result generalizes \cite[Lem.~2.12]{2026a_Jung}.

\begin{theorem} \label{thm:repr_sdd}
    Let $X$ be a normed vector space over $\mathbb{R}$ equipped with a norm $\| \cdot \|$.
    Let $f: \Omega \to \mathbb{R}$ be a function, where $\Omega$ is an open subset of $X$.
    Fix points $x \in \Omega$ and $v \in X \setminus \left\{ 0 \right\}$.
    If the specular directional derivative $\partial^{\sd}_v f(x)$ exists as a real number, then the following statements hold:
    \begin{enumerate}[label=\upshape(\alph*)]
        \item \label{thm:repr_sdd-1} If \ref{H1} holds, then \ref{H2} holds with 
        \begin{equation} \label{thm:repr_ldd}
            \partial^-_v f(x) 
            =
            \| v \|  \tan \left( 2 \arctan \left( \frac{\partial^{\sd}_v f(x)}{\| v \|} \right) - \arctan \left( \frac{\partial^+_v f(x)}{\| v \|} \right) \right). 
        \end{equation}        
        \item \label{thm:repr_sdd-2} If \ref{H2} holds, then \ref{H1} holds with 
        \begin{displaymath}
            \partial^+_v f(x) 
            =
            \| v \| \tan \left( 2 \arctan \left( \frac{\partial^{\sd}_v f(x)}{\| v \|} \right) - \arctan \left( \frac{\partial^-_v f(x)}{\| v \|} \right) \right).  
        \end{displaymath}   
        \item \label{thm:repr_sdd-3} If \ref{H1} or \ref{H2} holds, then \ref{H3} holds.   
    \end{enumerate}
    Conversely, if \ref{H1}, \ref{H2}, and \ref{H3} hold, then $\partial^{\sd}_v f(x)$ exists as a real number with the formula
    \begin{equation}    \label{eq:repr_sdd}
        \partial_{v}^{\sd}f(x) 
        =
        \mathcal{B}_{\|v\|} \left( \partial^+_v f(x), \partial^-_v f(x) \right),
    \end{equation}
    with the interpretation $\arctan(\pm \infty) = \pm \frac{\pi}{2}$.
\end{theorem}

\begin{proof}
    Since $\Omega$ is open, choose $\varepsilon>0$ such that $x+tv\in\Omega$ for every $t\in(-\varepsilon,\varepsilon)$, and define $F:(-\varepsilon,\varepsilon)\to\mathbb{R}$ by
    \begin{displaymath}
        F(t):=\frac{f(x+tv)}{\|v\|}.
    \end{displaymath}
    For $t \in (-\varepsilon, \varepsilon)$, write 
    \begin{displaymath}
        a := f(x + tv) - f(x), \qquad
        b := f(x) - f(x - tv), \qquad
        r := \| v \|.
    \end{displaymath}
    By the definitions of the one-sided derivatives,
    \begin{displaymath}
        \partial^+F(0)=\frac{\partial_v^+f(x)}{r}
        \qquad\text{and}\qquad
        \partial^-F(0)=\frac{\partial_v^-f(x)}{r},
    \end{displaymath}
    whenever the corresponding limits exist in $\overline{\mathbb{R}}$.
    Moreover, by \eqref{eq:A_subscript}, we obtain 
    \begin{align*}
        \partial_v^{\sd}f(x)
        &= \lim_{t \searrow 0} \mathcal{A}_r\bigl(a, b, t\bigr)     \\
        &= \lim_{t \searrow 0} r\mathcal{A}\bigl(F(t)-F(0), F(0)-F(-t), t\bigr)   \\
        &= r F^{\sd}(0),
    \end{align*}
    whenever either limit exists.
    Consequently, $\partial_v^{\sd}f(x)$ exists as a real number if and only if $F^{\sd}(0)$ does.

    First, suppose that $\partial_v^{\sd}f(x)$ exists as a real number and that \ref{H1} or \ref{H2} holds.
    Then $F^{\sd}(0)$ exists as a real number and at least one of $\partial^+F(0)$ and $\partial^-F(0)$ exists in $\overline{\mathbb{R}}$.
    By \cite[Lem.~2.12]{2026a_Jung} applied to $F$ at $t=0$, both $\partial^+F(0)$ and $\partial^-F(0)$ exist in $\overline{\mathbb{R}}$ and they cannot both equal $+\infty$ or both equal $-\infty$, proving \cref{thm:repr_sdd-3}.
    By \cite[Thm.~2.14]{2026a_Jung},
    \begin{displaymath}
        F^{\sd}(0) = \mathcal{B} \left( \partial^+ F(0), \partial^- F(0) \right),
    \end{displaymath}
    which is equivalent to 
    \begin{displaymath}
        2\arctan\left(\frac{\partial_v^{\sd}f(x)}{r}\right)
        =\arctan\left(\frac{\partial_v^+f(x)}{r}\right)
        +\arctan\left(\frac{\partial_v^-f(x)}{r}\right).
    \end{displaymath}
    Solving for either one-sided arctangent, applying the extended tangent function with the conventions $\arctan(\pm\infty) = \pm\frac{\pi}{2}$ and $\tan(\pm\frac{\pi}{2}) = \pm\infty$, and multiplying by $r$ give the formulas in \cref{thm:repr_sdd-1,thm:repr_sdd-2}.

    Conversely, suppose that \ref{H1}, \ref{H2}, and \ref{H3} hold.
    These imply 
    \begin{displaymath}
        \theta
        :=
        \frac{\arctan \bigl( \partial^+F(0) \bigr) + \arctan \bigl( \partial^-F(0) \bigr)}{2}
        \in
        \left(-\frac{\pi}{2}, \frac{\pi}{2}\right).
    \end{displaymath}
    Using the convention $\arctan(\pm\infty) = \pm\frac{\pi}{2}$, we obtain
    \begin{displaymath}
        \lim_{h \searrow 0} \left[ \frac{1}{2}\arctan\left(\frac{F(h)-F(0)}{h}\right) + \frac{1}{2}\arctan\left(\frac{F(0)-F(-h)}{h}\right) \right] 
        =
        \theta.
    \end{displaymath}
    Since $\tan$ is continuous at $\theta$, we have 
    \begin{displaymath}
        \lim_{h \searrow 0} \mathcal{B} \left( \frac{F(h)-F(0)}{h}, \frac{F(0)-F(-h)}{h}\right) 
        =
        \tan \theta.
    \end{displaymath}
    By \eqref{lem:rep_of_sdd-2}, $F^{\sd}(0)$ exists as a real number and equals $\tan\theta$.
    Therefore, $\partial_v^{\sd}f(x)$ also exists as a real number, and
    \begin{displaymath}
        \partial_v^{\sd}f(x)
        =rF^{\sd}(0)
        =r\tan\theta
        =\mathcal{B}_r\bigl(\partial_v^+f(x),\partial_v^-f(x)\bigr),
    \end{displaymath}
    which proves \eqref{eq:repr_sdd}.
\end{proof}

We provide three corollaries of \cref{thm:repr_sdd}.
First, the following corollary gives a sufficient condition under which multiplying a function by a scalar preserves specular directional differentiability.

\begin{corollary} \label{cor:multiplied_sdd}
    Let $X$ be a normed vector space over $\mathbb{R}$ equipped with a norm $\| \cdot \|$.
    Let $f: \Omega \to \mathbb{R}$ be a function, where $\Omega$ is an open subset of $X$.
    Fix points $x \in \Omega$ and $v \in X$.
    If \ref{H1}, \ref{H2}, and \ref{H3} hold, then $\partial_v^{\sd} (\lambda f) (x)$ exists for all $\lambda \in \mathbb{R}$.
\end{corollary}

\begin{proof}
    Fix $\lambda \in \mathbb{R}$.
    The conclusion immediately follows if $\lambda = 0$ or $v = 0$.
    Thus, assume that $\lambda \neq 0$ and $v \neq 0$.
    Since \ref{H1}, \ref{H2}, and \ref{H3} hold for $f$, \ref{H1}, \ref{H2}, and \ref{H3} hold for $\lambda f$.
    Therefore, \cref{thm:repr_sdd} ensures the existence of the specular directional derivative $\partial^{\sd}_v (\lambda f) (x)$.
\end{proof}

In \cref{cor:multiplied_sdd}, the assumptions \ref{H1} and \ref{H2} cannot both be dropped.

\begin{example}
    Define $f:\left( -\frac{1}{2}, \frac{1}{2} \right) \to\mathbb{R}$ by
    \begin{displaymath}
        f(t)=
        \begin{cases}
            0, & \displaystyle  t \in \left\{0\right\} \cup \bigcup_{i=1}^{\infty} \left(\frac{1}{2i+1},\frac{1}{2i}\right),\\[1em]
            t, & \displaystyle  t \in \bigcup_{i = 1}^{\infty} \left(-\frac{1}{2i},-\frac{1}{2i+1}\right), \\[1em]
            \left( \tan\frac{\pi}{8} \right) t, & \text{otherwise},
        \end{cases}
    \end{displaymath}
    for $t \in \left( -\frac{1}{2}, \frac{1}{2} \right)$.
    For $\lambda \in \mathbb{R}$, define the function $F_{\lambda} : \left(0, \frac{1}{2}\right) \to \mathbb{R}^2$ by 
    \begin{displaymath}
        F_{\lambda}(h) := \left( \frac{\lambda f(h) - \lambda f(0)}{h}, \frac{\lambda f(0) - \lambda f(-h)}{h} \right).
    \end{displaymath}
    On the one hand, consider $\lambda = 1$.
    Then, neither $\partial^+ f(0)$ nor $\partial^- f(0)$ exists since $F_1(h)$ takes each of the two values $(0, 1)$ and $\left( \tan\frac{\pi}{8}, \tan\frac{\pi}{8} \right)$ arbitrarily close to $0$.
    However, the two possible values give the same value under $\mathcal{B}$:
    \begin{displaymath}
        \mathcal{B}(0,1)
        = \tan\left(\frac{1}{2}\arctan 1\right)
        = \tan\frac{\pi}{8}
        = \mathcal{B}\left( \tan\frac{\pi}{8}, \tan\frac{\pi}{8} \right),
    \end{displaymath}
    and hence $f^{\sd}(0)$ exists.
    On the other hand, consider $\lambda = 2$.
    For the function $2f$, $F_{2}(h)$ takes each of the two values $(0, 2)$ and $\left( 2 \tan\frac{\pi}{8}, 2 \tan\frac{\pi}{8} \right)$ arbitrarily close to $0$.
    Since 
    \begin{displaymath}
        \mathcal{B}(0,2)
        = \tan\left(\frac{1}{2}\arctan 2\right)
        \neq 2 \tan\frac{\pi}{8}
        = \mathcal{B}\left( 2\tan\frac{\pi}{8}, 2\tan\frac{\pi}{8} \right),
    \end{displaymath}
    the specular derivative $(2f)^{\sd}(0)$ does not exist.
\end{example}

Second, once the bounds of the one-sided directional derivatives are known, the specular directional derivative can be estimated accordingly.

\begin{corollary} \label{cor:estimate_of_sdd}
    Let $X$ be a normed vector space over $\mathbb{R}$ equipped with a norm $\| \cdot \|$.
    Let $f: \Omega \to \mathbb{R}$ be a function, where $\Omega$ is an open subset of $X$.
    Fix points $x \in \Omega$ and $v \in X \setminus \left\{ 0 \right\}$.
    Suppose that \ref{H1}, \ref{H2}, and \ref{H3} hold.
    Let $m_1$ and $m_2$ be extended real numbers.
    Then, the following statements hold.
    \begin{enumerate}[label=\upshape(\alph*), itemsep=1ex]
        \item \label{cor:estimate_of_sdd-1} It holds that either $\partial^{-}_v f(x) \leq \partial^{\sd}_v f(x) \leq \partial^{+}_v f(x)$ or $\partial^{+}_v f(x) \leq  \partial^{\sd}_v f(x) \leq \partial^{-}_v f(x)$.
        \item \label{cor:estimate_of_sdd-2} If $m_1 \leq \partial^{+}_v f(x)$ and $m_2 \leq \partial^{-}_v f(x)$, then $\min\left\{ m_1, m_2 \right\} \leq \partial_{v}^{\sd}f(x)$.
        \item \label{cor:estimate_of_sdd-3} If $\partial^{+}_v f(x) \leq m_1$ and $\partial^{-}_v f(x) \leq m_2$, then $\partial_{v}^{\sd}f(x) \leq \max\left\{ m_1, m_2 \right\}$.
    \end{enumerate}
\end{corollary}

\begin{proof}
    By \cref{thm:repr_sdd}, $\partial^{\sd}_v f(x)$ exists as a real number with the formula \eqref{eq:repr_sdd}.
    
    We prove \cref{cor:estimate_of_sdd-1} first.
    If $\partial^{\pm}_v f(x)$ are finite, the conclusion follows from the monotonicity of the tangent function on $(-\frac{\pi}{2}, \frac{\pi}{2})$, dividing cases into $\partial^+_v f(x) \leq \partial^-_v f(x)$ and $\partial^+_v f(x) \geq \partial^-_v f(x)$.
    If one of $\partial_v^\pm f(x)$ is infinite and the other is its opposite infinity, then the conclusion is trivial.
    If, for example, $\partial^+_v f(x)=\infty$ and $\partial^-_v f(x) \in \mathbb{R}$, then
    \begin{displaymath}
        \arctan\left(\frac{\partial_v^-f(x)}{\|v\|}\right)
        \leq
        \frac{\pi}{4}+\frac12\arctan\left(\frac{\partial_v^-f(x)}{\|v\|}\right)
        < \frac{\pi}{2}
    \end{displaymath}
    since $\arctan t < \frac{\pi}{2}$ for all $t \in \mathbb{R}$.
    Thus, the monotonicity of the tangent function on $(-\frac{\pi}{2}, \frac{\pi}{2})$ and \eqref{eq:repr_sdd} imply that 
    \begin{displaymath}
        \frac{\partial_v^-f(x)}{\|v\|} \leq \tan\left( \frac{\pi}{4} + \frac{1}{2} \arctan\left(\frac{\partial_v^-f(x)}{\|v\|}\right)
        \right) 
        = \frac{\partial_v^{\sd}f(x)}{\|v\|}.
    \end{displaymath}
    Since $\partial_v^+f(x)=\infty$, this implies $\partial^{-}_v f(x) \leq \partial^{\sd}_v f(x) \leq \partial^{+}_v f(x)$.
    The remaining infinite cases are proved similarly.

    To show \cref{cor:estimate_of_sdd-2}, assume that $m_1 \leq \partial^{+}_v f(x)$ and $m_2 \leq \partial^{-}_v f(x)$.
    Note that $m_1 = m_2 = \infty$ is impossible by the condition \labelcref{H3}.
    Write $m := \min\left\{ m_1, m_2 \right\}$.
    Then, it follows that
    \begin{align*}
        2 \arctan\left( \dfrac{m}{\| v \|} \right)  
        &\leq \arctan\left( \dfrac{m_1}{\| v \|} \right) + \arctan \left( \dfrac{m_2}{\| v \|} \right)     \\
        &\leq \arctan \left( \dfrac{\partial^+_vf(x)}{\| v \|} \right) + \arctan \left( \dfrac{\partial^-_v f(x)}{\| v \|} \right).
    \end{align*}
    Dividing the above inequality by $2$, one can find that
    \begin{displaymath}
        \dfrac{m}{\| v \|} \leq \tan\left( \dfrac{1}{2} \arctan \left( \dfrac{\partial^+_vf(x)}{\| v \|} \right) + \dfrac{1}{2} \arctan \left( \dfrac{\partial^-_v f(x)}{\| v \|} \right) \right) 
        = \dfrac{\partial^{\sd}_v f(x)}{\| v \|},
    \end{displaymath}
    where the equality follows from the formula \eqref{eq:repr_sdd}.
    Multiplying both sides by $\| v \|$ yields the desired inequality of the second part.
    \Cref{cor:estimate_of_sdd-3} can be proved similarly.
\end{proof}

\begin{remark}
    The conclusions in \cref{cor:estimate_of_sdd-2,cor:estimate_of_sdd-3} of \cref{cor:estimate_of_sdd} can be strengthened to strict inequalities if at least one of the inequalities in the assumptions is strict. This is due to the strictly increasing nature of the arctangent function.
\end{remark}

Third, the formula \eqref{eq:repr_sdd} can be expressed in terms of the function $\mathcal{C}$ defined in \eqref{def:C}.

\begin{corollary} \label{cor:repr_sdd_C}
    Let $X$ be a normed vector space over $\mathbb{R}$ equipped with a norm $\| \cdot \|$.
    Let $f: \Omega \to \mathbb{R}$ be a function, where $\Omega$ is an open subset of $X$.
    Fix points $x \in \Omega$ and $v \in X \setminus \left\{ 0 \right\}$.
    If \ref{H1}, \ref{H2}, and \ref{H3} hold, then $\partial^{\sd}_v f(x)$ exists as a real number with the formula 
    \begin{equation}    \label{eq:cor:repr_sdd_C}
        \partial_{v}^{\sd} f(x)
        =
        \mathcal{C}_{\|v\|} \left( \partial^+_v f(x), \partial^-_v f(x) \right).
    \end{equation}
    Moreover, if $\partial^{\pm}_v f(x) \in \mathbb{R}$, then
    \begin{align}
        \partial_{v}^{\sd}f(x)
        &\geq
        \dfrac{\partial^+_v f(x) + \partial^-_v f(x) - |\partial^+_v f(x) + \partial^-_v f(x)|}{4}, \label{ineq:est_of_sdd-1}    \\
        \partial_{v}^{\sd}f(x)
        &\leq
        \dfrac{\partial^+_v f(x) + \partial^-_v f(x) + |\partial^+_v f(x) + \partial^-_v f(x)|}{4},     \label{ineq:est_of_sdd-2}   \\
        |\partial_{v}^{\sd}f(x)|
        &\leq
        \dfrac{|\partial^+_v f(x) + \partial^-_v f(x)|}{2}.    \label{ineq:abs_est_of_sdd}  
    \end{align}
\end{corollary}

\begin{proof}
    Observe that 
    \begin{equation} \label{eq: trigonometric identities - 1}
        \partial_{v}^{\sd}f(x) 
        = \| v \| \, \mathcal{B} \left( \dfrac{\partial^+_v f(x)}{\| v \|}, \dfrac{\partial^-_v f(x)}{\| v \|} \right)
        = \| v \| \, \mathcal{C} \left( \dfrac{\partial^+_v f(x)}{\| v \|}, \dfrac{\partial^-_v f(x)}{\| v \|} \right)
    \end{equation}
    by applying \cref{thm:repr_sdd} and the equality \cref{eq:BC}.
    This proves \eqref{eq:cor:repr_sdd_C}.
    The estimates \cref{ineq:abs_est_of_sdd,ineq:est_of_sdd-1,ineq:est_of_sdd-2} can be shown by applying \cite[Lem.~A.2~(h) and~(i)]{2026a_Jung}.
\end{proof}

\begin{example}
    Let $f, g: \Omega \to \mathbb{R}$ be functionals, where $\Omega$ is an open subset of a normed vector space $X$ over $\mathbb{R}$.
    Suppose that \ref{H1}, \ref{H2}, and \ref{H3} hold for the mapping $x \mapsto f(x) + g(x)$ for $x \in \Omega$.
    Given $x \in \Omega$ and $v \in X$, if $\partial_v^{\sd}f(x)$ and $\partial_v g(x)$ exist and $\partial_v^{\pm}f(x)\in\mathbb{R}$, then
    \begin{displaymath}
        \left\vert \partial_v^{\sd}(f(x) + g(x)) \right\vert \leq \dfrac{\left\vert \partial^+_v f(x) + \partial^-_v f(x) \right\vert }{2} + |\partial_v g(x)|
    \end{displaymath}    
    by \cref{cor:repr_sdd_C} and \cite[Lem.~A.2~(i)]{2026a_Jung}.
\end{example}

\section{Specular G\^ateaux derivatives}    \label{sec:sGd}

In this section, we impose stronger differentiability assumptions in the specular sense.
We prove the uniqueness of a specular G\^ateaux derivative and explore the relationships between specular and classical G\^ateaux differentiability, as summarized in \cref{fig:differentiability}.
Also, we state and prove the Quasi-Mean Value Theorem for real-valued specularly G\^ateaux differentiable functions. 

\subsection{Vector-valued functions}

First, a specular G\^ateaux derivative is unique if it exists.

\begin{proposition} \label{prop:uniqueness_sGd}
    Let $X$ and $Y$ be normed vector spaces over $\mathbb{R}$ equipped with norms $\| \cdot \|_X$ and $\| \cdot \|_Y$, respectively. 
    Let $f : \Omega \to Y$ be a function, where $\Omega$ is an open subset of $X$.
    If the specular G\^ateaux derivative exists at $x \in \Omega$, then it is unique.
\end{proposition}

\begin{proof}
    Suppose that $\ell_1, \ell_2 \in \mathcal{L}(X;Y)$ are two specular G\^ateaux derivatives of $f$ at $x$.
    Then, by the definition of the specular G\^ateaux derivative, for every $v \in X$, $\ell_1(v) = \partial_v^{\sd} f(x) = \ell_2(v)$.
    This implies that $(\ell_1-\ell_2)(v)=0$ for all $v \in X$.
    Therefore $\ell_1=\ell_2$ in $\mathcal{L}(X;Y)$, proving uniqueness.
\end{proof}

Next, specular G\^ateaux differentiability generalizes G\^ateaux differentiability.

\begin{proposition}   \label{prop:G_implies_spG}
    Let $X$ and $Y$ be normed vector spaces over $\mathbb{R}$ equipped with norms $\| \cdot \|_X$ and $\| \cdot \|_Y$, respectively.
    Let $f : \Omega \to Y$ be a function, where $\Omega$ is an open subset of $X$.
    If $f$ is G\^ateaux differentiable at $x \in \Omega$, then $f$ is specularly G\^ateaux differentiable at $x$ and $\sGd f(x) = df(x)$, where $df(x)$ denotes the G\^ateaux derivative of $f$ at $x$.
\end{proposition}

\begin{proof}
    Since $f$ is G\^ateaux differentiable at $x$, there exists a linear operator $\ell \in \mathcal{L}(X; Y)$ such that $\ell = df(x)$ and $\ell(v) = \partial_v f(x)$ for all $v \in X$.
    By \cref{prop:dd_implies_sdd}, we have $\partial_v f(x) = \partial_v^{\sd} f(x) = \ell(v)$ for all $v \in X$.
    Therefore, we obtain that $\ell = \sGd f(x)$, as required.
\end{proof}

Recall that classical G\^ateaux differentiability does not imply continuity; see \cite[Ex.~4.2.1]{2010_Borwein_BOOK}.
By \cref{prop:G_implies_spG}, the same is true of specular G\^ateaux differentiability.
Moreover, regular specular differentiability in every direction does not imply specular G\^ateaux differentiability.

\begin{example} \label{ex:rsd_does_not_implies_sGd}   
    Consider the function $f:\mathbb{R}^2 \to \mathbb{R}$ defined by 
    \begin{displaymath}
        f(x_1, x_2) := \max \left\{ x_1, x_2 \right\}.
    \end{displaymath}
    Note that $f$ is regularly specularly differentiable at the origin in every direction with 
    \begin{displaymath}
        \partial^{\sd}_{x_1} f(0, 0) = \sqrt{2} - 1 = \partial^{\sd}_{x_2} f(0, 0).
    \end{displaymath}
    However, since $f(t,t)=t$ for all $t\in\mathbb{R}$,
    \begin{displaymath}
        \partial_{(1,1)}^{\sd}f(0, 0)
        =
        1
        \neq
        2 \sqrt{2} - 2
        =
        \partial_{x_1}^{\sd}f(0, 0) + \partial_{x_2}^{\sd}f(0, 0).
    \end{displaymath}
    Thus, the map $v \mapsto \partial_v^{\sd}f(0, 0)$ is not additive, and hence $f$ is not specularly G\^ateaux differentiable at the origin.
\end{example}

The following proposition shows that perturbations of order $o(|t|)$ along each fixed direction do not change the specular G\^ateaux derivative.
To show this, we record the following lemma.

\begin{lemma} \label{lem:A_unit_step_Lipschitz}
    Let $Y$ be a normed vector space over $\mathbb{R}$ equipped with norm $\| \cdot \|_Y$.
    Then
    \begin{equation} \label{eq:A_unit_step_Lipschitz}
        \|\mathcal{A}^Y(a,b,1)-\mathcal{A}^Y(c,d,1)\|_Y
        \leq
        2\max\{\|a-c\|_Y, \|b-d\|_Y\}
    \end{equation}
    for all $a,b,c,d\in Y$.
\end{lemma}

\begin{proof}
    Write $M := \max \left\{ \left\| a - c \right\|_Y, \left\| b - d \right\|_Y \right\}$ and 
    \begin{displaymath}
        A := \|(1,a)\|_{\mathbb{R}\times Y},\quad
        B := \|(1,b)\|_{\mathbb{R}\times Y},\quad
        C := \|(1,c)\|_{\mathbb{R}\times Y},\quad
        D := \|(1,d)\|_{\mathbb{R}\times Y}.
    \end{displaymath}
    By the reverse triangle inequality and the triangle inequality, we have
    \begin{displaymath}   
        |A - C| \leq \|a-c\|_Y,
        \qquad
        |B - D| \leq \|b-d\|_Y,
        \qquad
        \|c - d\|_Y \leq C + D.
    \end{displaymath}
    These imply
    \begin{equation}   \label{ineq:A_unit_step_Lipschitz-1}
        \|c - d\|_Y |AD - BC| 
        =
        \|c - d\|_Y |A(D - B) + B(A - C)| 
        \leq 
        (A + B) (C + D) M.
    \end{equation}
    The definition of $\mathcal{A}^Y$ gives
    \begin{align*}
        \mathcal{A}^Y(a,b,1)-\mathcal{A}^Y(c,d,1)
        &= \frac{bA+aB}{A+B}-\frac{dC+cD}{C+D}\\
        &=\frac{(b-d)A + (a-c)B}{A+B} - \frac{(c-d)(AD - BC)}{(A+B)(C+D)}.
    \end{align*}
    Combining this with \eqref{ineq:A_unit_step_Lipschitz-1} implies 
    \begin{align*}
        \|\mathcal{A}^Y(a,b,1)-\mathcal{A}^Y(c,d,1)\|_Y
        &\leq \frac{A\|b - d\|_Y + B\|a - c\|_Y}{A + B} + \frac{\|c - d\|_Y |AD - BC|}{(A+B)(C+D)}\\
        &\leq 2 M,
    \end{align*}
    proving \eqref{eq:A_unit_step_Lipschitz}.
\end{proof}

\begin{proposition}
    \label{prop:sGd_higher_order_stability}
    Let $X$ and $Y$ be normed vector spaces over $\mathbb{R}$ equipped with norms $\| \cdot \|_X$ and $\| \cdot \|_Y$, respectively.
    Let $\Omega$ be an open subset of $X$.
    Let $f,g:\Omega\to Y$ and fix $x\in\Omega$.
    Suppose that
    \begin{equation}\label{eq:sGd_higher_order_stability}
        \|(f-g)(x+tv)-(f-g)(x)\|_Y = o(|t|)
    \end{equation}
    as $t \to 0$ for every fixed $v \in X$.
    Then $f$ is specularly G\^ateaux differentiable at $x$ if and only if $g$ is specularly G\^ateaux differentiable at $x$.
    In this case, $\sGd f(x)=\sGd g(x)$.
\end{proposition}

\begin{proof}
    Choose $\rho>0$ such that $B_\rho(x)\subset\Omega$.
    Fix $v\in X\setminus\{0\}$ and write $r:=\|v\|_X$.
    For $|t| < r^{-1} \rho$, define
    \begin{displaymath}
        R(t):=(f-g)(x+tv)-(f-g)(x).
    \end{displaymath}
    For $u \in \{f, g\}$ and $0 < t < r^{-1} \rho$, define
    \begin{displaymath}
        Q_u(t) := \mathcal{A}_r^Y\bigl(u(x+tv)-u(x),u(x)-u(x-tv),t\bigr).
    \end{displaymath}
    Using \eqref{eq:A_subscript} and \eqref{eq:A_multiplied_arguments}, we apply \cref{lem:A_unit_step_Lipschitz} to obtain
    \begin{align*}
        \|Q_f(t)-Q_g(t)\|_Y
        &\leq 
        2r \max\left\{ \left\|\frac{R(t)}{tr}\right\|_Y, \left\|\frac{-R(-t)}{tr}\right\|_Y \right\} \\
        &=
        \frac{2}{t}\max\{\|R(t)\|_Y, \|R(-t)\|_Y\} \\
        &\to
        0
    \end{align*}
    as $t\searrow0$, where the convergence follows from \eqref{eq:sGd_higher_order_stability}.

    Suppose that $f$ is specularly G\^ateaux differentiable at $x$.
    By \eqref{eq:sdd},
    \begin{displaymath}
        Q_f(t) 
        \to 
        \partial_v^{\sd}f(x)
        =
        \sGd f(x)(v)
    \end{displaymath}
    as $t\searrow0$.
    The triangle inequality gives
    \begin{displaymath}
        \|Q_g(t)-\sGd f(x)(v)\|_Y
        \leq
        \|Q_g(t)-Q_f(t)\|_Y + \|Q_f(t)-\sGd f(x)(v)\|_Y
        \to
        0
    \end{displaymath}
    as $t \searrow 0$.
    Hence, $\partial_v^{\sd}g(x)=\sGd f(x)(v)$.
    Since $v\in X\setminus\{0\}$ was arbitrary, we obtain
    \begin{displaymath}
        \partial_v^{\sd}g(x)
        =
        \partial_v^{\sd}f(x)
        =
        (\sGd f(x))(v)
    \end{displaymath}
    for every nonzero $v\in X$.
    Thus, the specular directional derivatives of $g$ at $x$ are represented by the bounded linear operator $\sGd f(x)$.
    Therefore, $g$ is specularly G\^ateaux differentiable at $x$ with $\sGd g(x)=\sGd f(x)$.
    Interchanging $f$ and $g$ proves the converse.
\end{proof}

\subsection{Real-valued functions}

In this section, we focus on the case $Y = \mathbb{R}$.
To generalize the Mean Value Theorem, we need the following one-dimensional form of the Quasi-Mean Value Theorem with specular derivatives.

\begin{remark}  \label{rmk:QMVT}
    The Quasi-Mean Value Theorem \cite[Thm.~3.3]{2026a_Jung} in $\mathbb{R}$ can be reformulated as follows.
    Let $f : [a, b] \to \mathbb{R}$ be continuous on the interval $[a, b]$, and suppose that $f^{\sd}(x)$ exists for all $x \in (a, b)$.
    Then, there exist $t_1, t_2 \in (0, 1)$ such that 
    \begin{displaymath}
        f^{\sd}(a + t_1(b - a)) (b - a) \leq f(b) - f(a) \leq f^{\sd}(a + t_2(b - a)) (b - a).
    \end{displaymath}    
\end{remark}

Now, we are ready to prove the Quasi-Mean Value Theorem in the specular sense for real-valued functions.

\begin{theorem}
    [Quasi-Mean Value Theorem] \label{thm:QMVT_nvs_real_valued}
    Let $X$ be a normed vector space over $\mathbb{R}$ equipped with a norm $\| \cdot \|$. 
    Suppose that $f : \Omega \to \mathbb{R}$ is continuous and specularly G\^ateaux differentiable in $\Omega$, where $\Omega$ is an open convex subset of $X$.
    Then, for all $u, v \in \Omega$ with $u \neq v$, it holds that 
    \begin{displaymath}
        \inf_{t \in (0, 1)} \partial_w^{\sd} f(u + tw)
        \leq f(v) - f(u)
        \leq \sup_{t \in (0, 1)} \partial_w^{\sd} f(u + tw),
    \end{displaymath}   
    where $w := v - u \in X \setminus \left\{ 0 \right\}$.
\end{theorem}

\begin{proof}
    Let $u,v\in\Omega$ be such that $u\neq v$, and write $w := v-u \in X \setminus \left\{ 0 \right\}$.
    Since $\Omega$ is convex, $u+tw\in\Omega$ for all $t\in[0, 1]$.
    Define $F:[0, 1]\to\mathbb{R}$ by
    \begin{displaymath}
        F(t) := \frac{f(u + tw)}{\|w\|}
    \end{displaymath}
    for $t \in [0, 1]$.
    Since $f$ is continuous, $F$ is continuous on $[0, 1]$.

    We first show that, for every $t \in (0, 1)$,
    \begin{equation}    \label{eq:QMVT_nvs_real_valued-1}
        F^{\sd}(t)
        =
        \frac{\partial_w^{\sd}f(u + tw)}{\|w\|}.
    \end{equation}
    Fix $t \in (0, 1)$ and define $x_t := u + tw$.
    Since $f$ is specularly G\^ateaux differentiable at $x_t$, the specular directional derivative $\partial_w^{\sd}f(x_t)$ exists.
    For every sufficiently small $h>0$, the identities \eqref{eq:A_multiplied_arguments} and \eqref{eq:A_subscript} imply
    \begin{align*}
        & \,\mathcal{A}\bigl(F(t+h)-F(t), F(t)-F(t-h), h \bigr) \\
        =& \, \mathcal{A}\left(\frac{f(x_t + hw) - f(x_t)}{\|w\|}, \frac{f(x_t) - f(x_t - hw)}{\|w\|}, h \right)   \\
        =& \, \mathcal{A}\bigl( f(x_t + hw) - f(x_t), f(x_t) - f(x_t - hw), h \|w\| \bigr)   \\
        \to& \, \frac{\partial_w^{\sd}f(x_t)}{\|w\|}
    \end{align*}
    as $h \searrow 0$, where the convergence follows from \eqref{eq:sdd}.
    Consequently, $F^{\sd}(t)$ exists and satisfies \eqref{eq:QMVT_nvs_real_valued-1}.

    By \cref{rmk:QMVT}, there exist
    $t_1, t_2 \in (0, 1)$ such that
    \begin{displaymath} 
        F^{\sd}(t_1) \leq F(1) - F(0) \leq F^{\sd}(t_2).
    \end{displaymath} 
    Combining these inequalities with \eqref{eq:QMVT_nvs_real_valued-1}, we obtain 
    \begin{displaymath}
        \partial_w^{\sd}f(u + t_1w)
        \leq f(v) - f(u)
        \leq \partial_w^{\sd}f(u + t_2w).
    \end{displaymath}
    Taking the infimum and supremum over $t \in (0, 1)$ gives the desired inequality.
\end{proof}

Specular derivatives need not vanish at local extrema; see \cite{2026a_Jung}.
However, we can estimate specular directional derivatives at local extrema.

\begin{theorem}[Local-extremum bound]
    \label{thm:quasi-Fermat}
    Let $X$ be a normed vector space over $\mathbb{R}$ equipped with a norm $\| \cdot \|$.
    Let $f : \Omega \to \mathbb{R}$ be specularly G\^ateaux differentiable in $\Omega$, where $\Omega$ is an open subset of $X$.
    If $x^{\ast} \in \Omega$ is a local minimizer or maximizer of $f$, then, for each $v \in X$, it holds that
    \begin{displaymath}
        \left\vert \partial_v^{\sd}f(x^{\ast}) \right\vert \leq \| v \|.
    \end{displaymath}
\end{theorem}

\begin{proof}
    Since the case $v = 0$ is obvious, we may assume that $v \in X$ with $v \neq 0$.
    First, assume that $x^{\ast}$ is a local minimizer of $f$.
    Let $h > 0$ be sufficiently small such that $x^{\ast} \pm hv \in \Omega$.
    Since $x^{\ast}$ is a minimizer, we have 
    \begin{displaymath}
      0 \leq \dfrac{1}{2}\arctan \left( \dfrac{f(x^{\ast} + hv) - f(x^{\ast})}{h \| v \| } \right) \leq \dfrac{\pi}{4}
    \end{displaymath}
    and 
    \begin{displaymath}
      - \dfrac{\pi}{4} \leq \dfrac{1}{2}\arctan \left( \dfrac{f(x^{\ast}) - f(x^{\ast} - hv)}{h \| v \| } \right) \leq 0.
    \end{displaymath}
    Adding the two inequalities, applying the increasing function $\tan$, multiplying by $\|v\|$, and letting $h \searrow 0$ yield the conclusion by \eqref{lem:rep_of_sdd-2}.
    The other case when $x^{\ast}$ is a maximizer can be shown similarly.
\end{proof}

\subsection{Subdifferential}

The specular G\^ateaux derivative of a semi-convex function belongs to its Fr\'echet subdifferential.

\begin{theorem} \label{thm:sGd_is_subd}
    Let $X$ be a normed vector space over $\mathbb{R}$ equipped with a norm $\| \cdot \|$.
    Let $\lambda \in \mathbb{R}$, and let $f:X\to\overline{\mathbb{R}}$ be a proper and $\lambda$-convex function.
    Let $\Omega$ be the interior of $\dom{f}$, and let $x \in \Omega$.
    If the restriction $f|_{\Omega}$ is specularly G\^ateaux differentiable at $x$, then 
    \begin{displaymath}
        f(w)
        \geq
        f(x) + \langle\sGd f(x), w - x\rangle + \frac{\lambda}{2}\|w - x\|_X^2
    \end{displaymath}
    for all $w \in X$.
    In particular, $\sGd f(x)\in\widehat{\partial}f(x)$.
    If $\lambda\geq0$, then $\sGd f(x)\in\partial f(x)$.
\end{theorem}

\begin{proof}
    Let $\ell:=\sGd f(x) \in X^{\ast}$.
    Fix $w \in \dom(f) \setminus \{x\}$, and set $v := w - x$ and $r := \|v\|>0$.
    Since $\Omega$ is open and $x\in\Omega$, choose $\varepsilon \in (0, 1)$ such that $x \pm tv \in \Omega$ whenever $0 < t < \varepsilon$.
    For such $t$, write 
    \begin{displaymath}
        \alpha_t:=\frac{f(x+tv)-f(x)}{tr} 
        \qquad\text{and}\qquad
        \beta_t:=\frac{f(x)-f(x-tv)}{tr}.
    \end{displaymath}
    From the $\lambda$-convexity of $f$, we have 
    \begin{displaymath} 
        f(x)\leq\frac{f(x+tv)+f(x-tv)}{2} - \frac{\lambda}{2} t^2 r^2,
    \end{displaymath}
    which implies 
    \begin{equation}    \label{ineq:sGd_is_subd-1}
        \beta_t 
        \leq 
        \alpha_t - \lambda tr.
    \end{equation}
    Moreover, since $x + tv = (1 - t)x +  tw$, the $\lambda$-convexity of $f$ implies
    \begin{displaymath}
        f(x+tv)\leq(1-t)f(x)+tf(w) - \frac{\lambda}{2} t(1 - t) r^2,
    \end{displaymath}
    which yields
    \begin{equation} \label{ineq:sGd_is_subd-2}
        \alpha_t
        \leq 
        \frac{f(w) - f(x)}{r} - \frac{\lambda}{2}(1 - t)r.
    \end{equation}
    Hence, by \cite[Lem.~A.2 (d)]{2026a_Jung}, \eqref{ineq:sGd_is_subd-1}, and \eqref{ineq:sGd_is_subd-2}, we obtain
    \begin{align*}
        \mathcal{C}(\alpha_t, \beta_t)
        &\leq \max \left\{ \alpha_t, \beta_t \right\}   \\
        &\leq \alpha_t + \max\left\{ 0, -\lambda \right\} tr  \\
        &\leq \frac{f(w) - f(x)}{r} - \frac{\lambda r}{2} + \frac{|\lambda| tr}{2}.
    \end{align*}
    Applying \cref{lem:rep_of_sdd} to $f|_{\Omega}$ and using
    $\ell(v)=\partial_v^{\sd}f(x)$, we obtain
    \begin{displaymath}
        \frac{\langle\ell,v\rangle}{r}
        =
        \frac{\partial_v^{\sd}f(x)}{r}
        =
        \frac{1}{r} \lim_{t \searrow 0} \mathcal{C}_r(r\alpha_t, r\beta_t)
        =
        \lim_{t \searrow 0} \mathcal{C}(\alpha_t,\beta_t)
        \leq
        \frac{f(w) - f(x)}{r} - \frac{\lambda}{2}r.
    \end{displaymath}
    Hence,
    \begin{displaymath}
        f(w) \geq f(x)+\langle\ell,w-x\rangle + \frac{\lambda}{2}\|w-x\|_X^2.
    \end{displaymath}
    The same inequality is immediate for $w=x$ and for $w \notin \dom(f)$.
    Consequently,
    \begin{equation} \label{ineq:sGd_is_subd-3}
        \frac{f(w) - f(x)-\langle\ell, w - x\rangle}{\|w - x\|_X}
        \geq 
        \frac{\lambda}{2} \|w - x\|_X
    \end{equation}
    for all $w \neq x$.
    
    Let $\varepsilon > 0$ be arbitrary, and choose $\delta = (1 + |\lambda|)^{-1} 2\varepsilon$.
    Then, for any $w \in X $ with $0 < \left\| w - x \right\|_X < \delta$, \eqref{ineq:sGd_is_subd-3} implies 
    \begin{displaymath}
        \frac{f(w) - f(x)-\langle\ell, w - x\rangle}{\|w - x\|_X}
        \geq 
        - \frac{|\lambda|}{2} \|w - x\|_X
        >
        - \varepsilon .
    \end{displaymath}
    Therefore, we conclude that $\sGd f(x) = \ell\in\widehat{\partial}f(x)$.

    If $\lambda \geq 0$, then $f$ is convex and \eqref{ineq:sGd_is_subd-3} yields $\ell\in\partial f(x)$.
\end{proof}

\section{Specular Fr\'echet differentials}   \label{sec:sFd}

This section is devoted to specular Fr\'echet differentiability, which is a stronger notion of differentiability than specular G\^ateaux differentiability.
We prove the uniqueness of specular Fr\'echet differentials and explore their relationships with other notions of differentiability, as summarized in \cref{fig:differentiability}.
We characterize specular Fr\'echet differentiability by uniform convergence of normalized specular difference quotients on the unit sphere and prove its equivalence with specular G\^ateaux differentiability for locally Lipschitz mappings defined on open subsets of $\mathbb{R}^n$.
Finally, we describe specular gradients in Hilbert spaces through orthonormal basis expansions and bounds at local extrema.

Observe that \eqref{eq:sFd} can be rewritten as 
\begin{multline}  
    \lim_{\left\| w \right\|_X \to 0} \left\| \left( \dfrac{f(x + w) - f(x) - \ell(w)}{\left\| w \right\|_X} \right) \dfrac{\left\| J \right\|_{X \times Y}}{\left\| J \right\|_{X \times Y} + \left\| K \right\|_{X \times Y}}\right. \\
    \left. +  \left( \dfrac{f(x) - f(x - w) - \ell(w)}{\left\| w \right\|_X} \right) \dfrac{\left\| K \right\|_{X \times Y}}{\left\| J \right\|_{X \times Y} + \left\| K \right\|_{X \times Y}}  \right\|_Y = 0,
\end{multline}
where 
\begin{equation}   \label{def: points JK}
    J := (w, f(x) - f(x - w)) 
    \qquad\text{and}\qquad
    K := (w, f(x + w) - f(x))
\end{equation}
for $w \in X$.

\subsection{Vector-valued functions}

A specular Fr\'echet differential is unique if it exists.

\begin{proposition} \label{prop:uniqueness_sFd}
    Let $X$ and $Y$ be normed vector spaces over $\mathbb{R}$ equipped with norms $\| \cdot \|_X$ and $\| \cdot \|_Y$, respectively. 
    Let $f : \Omega \to Y$ be a function, where $\Omega$ is an open subset of $X$.
    If the specular Fr\'echet differential exists at $x \in \Omega$, then it is unique.
\end{proposition}

\begin{proof}
    Let $\varepsilon > 0$ be arbitrary.
    Suppose that there exist two specular Fr\'echet differentials $\ell_1$ and $\ell_2$ of $f$ at $x \in \Omega$.
    Then, there exists $\delta > 0$ such that, if $0 < \left\| w \right\|_X < \delta$, then 
    \begin{displaymath}
        \dfrac{\left\| (f(x + w) - f(x) - \ell_1(w)) \left\| J \right\|_{X \times Y} + (f(x) - f(x - w) - \ell_1(w)) \left\| K \right\|_{X \times Y} \right\|_Y}{\left\| w \right\|_X \left\| J \right\|_{X \times Y} + \left\| w \right\|_X \left\| K \right\|_{X \times Y}} < \dfrac{\varepsilon}{2} 
    \end{displaymath}
    and 
    \begin{displaymath}
        \dfrac{\left\| (f(x + w) - f(x) - \ell_2(w)) \left\| J \right\|_{X \times Y} + (f(x) - f(x - w) - \ell_2(w)) \left\| K \right\|_{X \times Y} \right\|_Y}{\left\| w \right\|_X \left\| J \right\|_{X \times Y} + \left\| w \right\|_X \left\| K \right\|_{X \times Y}} < \dfrac{\varepsilon}{2} ,
    \end{displaymath}
    where $J$ and $K$ are defined as in \eqref{def: points JK}.
    Thus, if $0 < \left\| w \right\|_X < \frac{\delta}{2}$, then
    \begin{align*}
        & \left\| (\ell_1 - \ell_2) (w) \right\|_Y \\
        \leq& \left\| \dfrac{(f(x + w) - f(x)) \left\| J \right\|_{X \times Y} + (f(x) - f(x - w)) \left\| K \right\|_{X \times Y}}{\left\| J \right\|_{X \times Y} + \left\| K \right\|_{X \times Y}} - \ell_1 (w)\right\|_Y    \\
        &+ \left\| \dfrac{(f(x + w) - f(x)) \left\| J \right\|_{X \times Y} + (f(x) - f(x - w)) \left\| K \right\|_{X \times Y}}{\left\| J \right\|_{X \times Y} + \left\| K \right\|_{X \times Y}} - \ell_2 (w) \right\|_Y    \\
        <& \, \varepsilon \left\| w \right\|_X.
    \end{align*}
    Decreasing $\delta$ if necessary, we obtain
    \begin{align*}
        \left\| \ell_1 - \ell_2 \right\|_{\mathcal{L}(X; Y)}
        &= \dfrac{1}{\delta} \cdot \sup_{\left\| x \right\|_X = \delta} \left\| (\ell_1 - \ell_2)(x) \right\| \\
        &\leq \dfrac{1}{\delta} \cdot \sup_{\left\| x \right\|_X = \delta} \left( \varepsilon \left\| x \right\|_X \right)
        \leq \dfrac{1}{\delta} \cdot \varepsilon \cdot \delta
        = \varepsilon.
    \end{align*}
    Since $\varepsilon > 0$ is arbitrary, we conclude that $\ell_1 = \ell_2$.
\end{proof}

Specular Fr\'echet differentiability generalizes Fr\'echet differentiability.

\begin{proposition}  \label{prop:sF_implies_F}
    Let $X$ and $Y$ be normed vector spaces over $\mathbb{R}$ equipped with norms $\| \cdot \|_X$ and $\| \cdot \|_Y$, respectively. 
    Let $f : \Omega \to Y$ be a function, where $\Omega$ is an open subset of $X$.
    If $f$ is classically Fr\'echet differentiable at $x \in \Omega$, then $f$ is specularly Fr\'echet differentiable at $x$ with $\sFd f(x) = Df(x)$.
\end{proposition}

\begin{proof}
    Let $\ell$ be the classical Fr\'echet differential of $f$ at $x$, that is, $\ell := Df(x) \in \mathcal{L}(X; Y)$.
    Choose $\rho > 0$ such that $B_\rho(x) \subset \Omega$.
    Let $w\in X$ satisfy $0 < \|w\|_X < \rho$, and write $r := \|w\|_X$.
    By the definition of $\mathcal{A}^Y$,
    \begin{displaymath}
        \mathcal{A}^Y\bigl(\ell(w), \ell(w), r \bigr)
        =
        \dfrac{\ell(w)}{r}.
    \end{displaymath}
    Apply the identity \eqref{eq:A_multiplied_arguments} and \cref{lem:A_unit_step_Lipschitz} to obtain
    \begin{align*}
        &\left\|\mathcal{A}^Y\bigl(f(x+w) - f(x), f(x) - f(x - w), r\bigr) - \frac{\ell(w)}{r}\right\|_Y \\
        =& \left\|\mathcal{A}^Y \left(\frac{f(x+w) - f(x)}{r}, \frac{f(x) - f(x - w)}{r}, 1 \right) - \mathcal{A}^Y\left(\frac{\ell(w)}{r}, \frac{\ell(w)}{r}, 1 \right) \right\|_Y \\
        \leq& \, \frac{2}{r} \max \bigl\{ \|f(x+w) - f(x) - \ell(w)\|_Y, \|f(x) - f(x-w) - \ell(w)\|_Y \bigr\} \\
        \to & \, 0
    \end{align*}
    as $\|w\|_X \to 0$.
    Thus, \eqref{eq:sFd} holds with $\ell = Df(x)$, completing the proof.
\end{proof}

Specular Fr\'echet differentiability implies specular G\^ateaux differentiability.

\begin{proposition} \label{prop:spF_implies_spG}
    Let $X$ and $Y$ be normed vector spaces over $\mathbb{R}$ equipped with norms $\| \cdot \|_X$ and $\| \cdot \|_Y$, respectively. 
    Let $f : \Omega \to Y$ be a function, where $\Omega$ is an open subset of $X$.
    If $f$ is specularly Fr\'echet differentiable at $x \in \Omega$, then $f$ is specularly G\^ateaux differentiable at $x$, and $\sGd f(x) = \sFd f(x)$.
\end{proposition}

\begin{proof}
    Let $v \in X \setminus \left\{ 0 \right\}$ be fixed. 
    Choose $w = hv$ for sufficiently small $h > 0$ in \eqref{eq:sFd} to obtain
    \begin{multline*}
      \lim_{h \searrow 0} \left\| \left( \dfrac{f(x + hv) - f(x) - h \, \sFd f(x) (v)}{h \| v \|_X} \right) \dfrac{\left\| U \right\|_{X \times Y}}{\left\| U \right\|_{X \times Y} + \left\| V \right\|_{X \times Y}}  \right. \\
       \left. + \left( \dfrac{f(x) - f(x - hv) - h \, \sFd f(x) (v)}{h \| v \|_X} \right) \dfrac{\left\| V \right\|_{X \times Y}}{\left\| U \right\|_{X \times Y} + \left\| V \right\|_{X \times Y}} \right\|_Y = 0,
    \end{multline*}
    where $U$ and $V$ are defined as in \eqref{def:pt_UV}.
    By multiplying both sides by $\| v \|_X$, we obtain
    \begin{multline*}
      \lim_{h \searrow 0} \left\| \left( \dfrac{f(x + hv) - f(x)}{h} \right) \dfrac{\left\| U \right\|_{X \times Y}}{\left\| U \right\|_{X \times Y} + \left\| V \right\|_{X \times Y}} \right. \\
      \left.+ \left( \dfrac{f(x) - f(x - hv)}{h} \right) \dfrac{\left\| V \right\|_{X \times Y}}{\left\| U \right\|_{X \times Y} + \left\| V \right\|_{X \times Y}} - \sFd f(x)(v) \right\|_Y = 0,
    \end{multline*}
    as required.
    Thus, $\partial_v^{\sd}f(x)=\sFd f(x)(v)$ for every $v \neq 0$.
    The same identity holds for $v=0$, since both sides are zero.
    Therefore, $\sGd f(x)=\sFd f(x)$.    
\end{proof}

Even for functions $f:\mathbb{R}\to\mathbb{R}$, neither specular G\^ateaux nor specular Fr\'echet differentiability implies continuity or Lebesgue measurability; see \cite[Ex.~4.2]{2026a_Jung}.

Specular Fr\'echet differentiability can be characterized by uniform convergence of normalized specular difference quotients on the unit sphere.
We denote the unit sphere of a normed vector space $X$ equipped with norm $\| \cdot \|_X$ by
\begin{displaymath}
    \mathbb{S}^X := \left\{ \xi \in X \, \middle| \, \left\| \xi \right\|_X = 1  \right\}.
\end{displaymath}
When $X = \mathbb{R}^n$ is equipped with the Euclidean norm $\|\cdot\|_{\mathbb{R}^n}$, we write $\mathbb{S}^{n-1}$ for $\mathbb{S}^X$.

\begin{theorem}   \label{thm:sFd_equivalence}
    Let $X$ and $Y$ be normed vector spaces over $\mathbb{R}$ equipped with norms $\| \cdot \|_X$ and $\| \cdot \|_Y$, respectively. 
    Assume that $X$ is nontrivial.
    Let $f : \Omega \to Y$ be a function, where $\Omega$ is an open subset of $X$.
    Fix $x\in\Omega$.
    The function $f$ is specularly Fr\'echet differentiable at $x$ if and only if there exists $\ell \in \mathcal{L}(X; Y)$ such that 
    \begin{equation}    \label{eq:sFd_equivalence-2}
        \lim_{r \searrow 0} \sup_{\xi \in \mathbb{S}^X} \left\| \, \mathcal{A}^Y \left( \frac{f(x + r \xi) - f(x)}{r}, \frac{f(x) - f(x - r \xi)}{r}, 1 \right) - \ell(\xi) \right\|_Y
        =
        0 .
    \end{equation}  
    In this case, the operator $\ell$ above is unique and equals $\sFd f(x)$.
\end{theorem}

\begin{proof}
    Choose $\rho>0$ such that $B_{\rho}(x) \subset \Omega$.
    For $\ell\in\mathcal{L}(X;Y)$ and $0 < \|w\|_X < \rho$, write
    \begin{displaymath}
        E_{\ell}(w)
        :=
        \left\| \mathcal{A}^Y\left( \frac{f(x+w)-f(x)}{\|w\|_X}, \frac{f(x)-f(x-w)}{\|w\|_X}, 1 \right) - \frac{\ell(w)}{\|w\|_X} \right\|_Y.
    \end{displaymath}

    Suppose that $f$ is specularly Fr\'echet differentiable at $x$.
    Then, by \eqref{eq:A_multiplied_arguments} and \eqref{eq:sFd}, the operator $\ell:=\sFd f(x)\in\mathcal{L}(X;Y)$ satisfies
    \begin{displaymath}
        \lim_{\|w\|_X\to0}E_{\ell}(w)=0.
    \end{displaymath}
    Given $\varepsilon>0$, choose $\delta\in(0, \min\left\{ \rho, \varepsilon \right\})$ such that $E_{\ell}(w) < \varepsilon$ whenever $0 < \|w\|_X < \delta$.
    For $r\in(0,\delta)$ and $\xi\in\mathbb{S}^X$, we have $\|r\xi\|_X=r$ and $\ell(r\xi)=r\ell(\xi)$.
    Hence,
    \begin{displaymath}
        0
        \leq 
        \left\| \, \mathcal{A}^Y\left( \frac{f(x+r\xi)-f(x)}{r}, \frac{f(x)-f(x-r\xi)}{r}, 1 \right) - \ell(\xi)
        \right\|_Y
        =
        E_{\ell}(r\xi)
        <
        \varepsilon.
    \end{displaymath}
    Taking the supremum over $\xi\in\mathbb{S}^X$ gives
    \begin{displaymath}
        0
        \leq
        \sup_{\xi\in\mathbb{S}^X} \left\| \, \mathcal{A}^Y\left( \frac{f(x+r\xi)-f(x)}{r}, \frac{f(x)-f(x-r\xi)}{r}, 1 \right) - \ell(\xi) \right\|_Y
        \leq
        \varepsilon.
    \end{displaymath}
    Since $\varepsilon>0$ is arbitrary and the above estimate holds for every $r > 0$ such that $r < \delta < \varepsilon$, \eqref{eq:sFd_equivalence-2} follows.

    Conversely, suppose that there exists $\ell\in\mathcal{L}(X;Y)$ satisfying \eqref{eq:sFd_equivalence-2}, which can be rewritten as
    \begin{equation} \label{eq:sFd_equivalence-3}
        \lim_{r\searrow0}\sup_{\xi\in\mathbb{S}^X}E_{\ell}(r\xi)=0.
    \end{equation}
    Let $w\in X$ be arbitrary with $0<\|w\|_X<\rho$.
    Write $r := \|w\|_X$ and $\eta := \frac{w}{r}$.
    Since $w = r\eta$, $\eta \in \mathbb{S}^X$, and $r\in(0,\rho)$, \eqref{eq:sFd_equivalence-3} implies 
    \begin{displaymath}
        0\leq E_{\ell}(w)
        =
        E_{\ell}(r\eta)
        \leq
        \sup_{\xi\in\mathbb{S}^X}E_{\ell}(r\xi)
        \to 0
    \end{displaymath}
    as $r = \|w\|_X \searrow 0$.
    Therefore, by \eqref{eq:A_multiplied_arguments} and \eqref{eq:sFd}, $f$ is specularly Fr\'echet differentiable at $x$ with $\sFd f(x)=\ell$.
    The uniqueness of $\ell$ follows from \cref{prop:uniqueness_sFd}.
\end{proof}

As an application of the above theorem, we can prove the converse of \cref{prop:spF_implies_spG} if $X = \mathbb{R}^n$ and the function is locally Lipschitz continuous.

\begin{theorem} \label{thm:lipschitz_sG_sF_Rn}
    Let $Y$ be a normed vector space over $\mathbb{R}$ equipped with norm $\| \cdot \|_Y$.
    Let $\Omega$ be an open subset of $\mathbb{R}^n$.
    Let $f:\Omega\to Y$ and fix $x\in\Omega$.
    Suppose that there exist $\rho,L>0$ such that $B_{\rho}(x) \subset \Omega$ and
    \begin{displaymath}
        \|f(y)-f(z)\|_Y \leq L\|y-z\|_{\mathbb{R}^n}
    \end{displaymath}
    for all $y, z \in B_{\rho}(x)$.
    Then $f$ is specularly G\^ateaux differentiable at $x$ if and only if it is specularly Fr\'echet differentiable at $x$.
    In this case, $\sGd f(x) = \sFd f(x)$.
\end{theorem}

\begin{proof}
    By \cref{prop:spF_implies_spG}, it suffices to prove that specular G\^ateaux differentiability implies specular Fr\'echet differentiability.
    Suppose that $f$ is specularly G\^ateaux differentiable at $x$, and write $\ell:=\sGd f(x)\in\mathcal{L}(\mathbb{R}^n;Y)$.
    For $0<r<\rho$ and $\xi\in\mathbb{S}^{n-1}$, define
    \begin{displaymath}
        Q_r(\xi)
        :=
        \mathcal{A}^Y\left( \frac{f(x+r\xi)-f(x)}{r}, \frac{f(x)-f(x-r\xi)}{r}, 1 \right).
    \end{displaymath}
    By \eqref{eq:A_multiplied_arguments} and the definition of the specular directional derivative,
    \begin{equation}    \label{eq:lipschitz_sG_sF_Rn-1}
        \lim_{r\searrow0}Q_r(\xi)
        =\partial_\xi^{\sd}f(x)
        =\ell(\xi)
    \end{equation}
    for every $\xi\in\mathbb{S}^{n-1}$.
    Moreover, the Lipschitz condition on $f$ gives
    \begin{displaymath}
        \frac{\|f(x\pm r\xi)-f(x\pm r\eta)\|_Y}{r}
        \leq
        L\|\xi-\eta\|_{\mathbb{R}^n}
    \end{displaymath}
    for all $\xi,\eta\in\mathbb{S}^{n-1}$.
    Applying \cref{lem:A_unit_step_Lipschitz} to the two pairs of normalized difference quotients yields
    \begin{equation}\label{eq:lipschitz_sG_sF_Rn-2}
        \|Q_r(\xi)-Q_r(\eta)\|_Y\leq2L\|\xi-\eta\|_{\mathbb{R}^n}
    \end{equation}
    for all $0<r<\rho$.

    We claim that $Q_r$ converges to $\ell$ uniformly on the sphere.
    Given $\varepsilon>0$, set
    \begin{displaymath}
        \delta:=\frac{\varepsilon}{3(2L+\|\ell\|_{\mathcal{L}(\mathbb{R}^n;Y)}+1)}.
    \end{displaymath}
    By the compactness of $\mathbb{S}^{n-1}$, there exist $N\in\mathbb{N}$ and $\xi_1,\ldots,\xi_N\in\mathbb{S}^{n-1}$ such that every $\xi\in\mathbb{S}^{n-1}$ satisfies $\|\xi - \xi_j\|<\delta$ for some $j\in\{1,2,\ldots,N\}$.
    From \eqref{eq:lipschitz_sG_sF_Rn-1}, for each $j = 1, 2, \ldots, N$, there exists $r_j \in (0, \rho)$ such that 
    \begin{equation}    \label{eq:lipschitz_sG_sF_Rn-3}
        \|Q_r(\xi_j)-\ell(\xi_j)\|_Y<\frac{\varepsilon}{3}
    \end{equation}
    whenever $0 < r < r_j$.
    Define $r_0:=\min\{r_1,\ldots,r_N\}$.
    Let $r \in (0, r_0)$ and $\xi \in \mathbb{S}^{n-1}$ be arbitrary.
    By the choice of $\xi_1,\ldots,\xi_N$, there exists $j\in\{1,2,\ldots,N\}$ such that $\|\xi-\xi_j\|<\delta$.
    By the linearity of $\ell$ and the operator norm inequality,
    \begin{equation}    \label{eq:lipschitz_sG_sF_Rn-4}
        \left\| \ell(\xi_j) - \ell(\xi) \right\|_Y 
        =
        \left\| \ell(\xi_j - \xi) \right\|_Y 
        \leq 
        \left\| \ell \right\|_{\mathcal{L}(\mathbb{R}^n; Y)} \left\| \xi_j - \xi \right\|_{\mathbb{R}^n}
        \leq 
        \left\| \ell \right\|_{\mathcal{L}(\mathbb{R}^n; Y)} \delta.
    \end{equation}    
    By \eqref{eq:lipschitz_sG_sF_Rn-2}, \eqref{eq:lipschitz_sG_sF_Rn-3}, and \eqref{eq:lipschitz_sG_sF_Rn-4},
    \begin{align*}
        \|Q_r(\xi)-\ell(\xi)\|_Y
        &\leq
        \|Q_r(\xi)-Q_r(\xi_j)\|_Y + \|Q_r(\xi_j)-\ell(\xi_j)\|_Y + \|\ell(\xi_j)-\ell(\xi)\|_Y\\
        &<
        \bigl(2L+\|\ell\|_{\mathcal{L}(\mathbb{R}^n;Y)}\bigr)\delta + \frac{\varepsilon}{3} \\
        &< 
        \frac{2\varepsilon}{3}.
    \end{align*}
    Taking the supremum over $\xi\in\mathbb{S}^{n-1}$ gives
    \begin{displaymath}
        \sup_{\xi\in\mathbb{S}^{n-1}}\|Q_r(\xi)-\ell(\xi)\|_Y
        \leq\frac{2\varepsilon}{3}<\varepsilon
    \end{displaymath}
    whenever $0 < r < r_0$.
    Since $\varepsilon>0$ is arbitrary, this proves the asserted uniform convergence as $r \searrow 0$.

    Applying \cref{thm:sFd_equivalence}, we conclude that $f$ is specularly Fr\'echet differentiable at $x$ with $\sFd f(x)=\ell=\sGd f(x)$.
\end{proof}

The local Lipschitz assumption in \cref{thm:lipschitz_sG_sF_Rn} cannot be replaced by continuity.

\begin{example}\label{ex:continuous_sG_not_sF}
    Define the function $f:\mathbb{R}^2\to\mathbb{R}$ by
    \begin{displaymath}
        f(x_1, x_2) 
        :=
        \begin{cases}
            \dfrac{x_1^4 x_2}{x_1^6 + x_2^2}    &    \mbox{if } (x_1, x_2) \neq (0, 0),    \\[1em]
            0    &    \mbox{if } (x_1, x_2) = (0, 0).
        \end{cases}
    \end{displaymath}
    One can show that $f$ is continuous and classically G\^ateaux differentiable at the origin, with $d f(0,0)=0$, where $d f(0,0)$ denotes the classical G\^ateaux derivative of $f$ at the origin.
    Thus, \cref{prop:G_implies_spG} yields $\sGd f(0,0)=0$.

    We claim that $f$ is not specularly Fr\'echet differentiable at the origin.
    If $f$ were specularly Fr\'echet differentiable at the origin, \cref{prop:spF_implies_spG} would imply $\sFd f(0,0)=0$.
    Since $f(-w) = - f(w)$ for every $w\in\mathbb{R}^2$, \eqref{eq:sFd} would then require
    \begin{displaymath}
        \left\vert \mathcal{A} \bigl( f(w), -f(-w), \left\| w \right\|_{\mathbb{R}^2} \bigr) \right\vert
        =
        \left\vert \mathcal{A} \bigl( f(w), f(w), \left\| w \right\|_{\mathbb{R}^2} \bigr) \right\vert
        =
        \frac{|f(w)|}{\left\| w \right\|_{\mathbb{R}^2}} 
        \to 
        0
    \end{displaymath}
    as $\left\| w \right\|_{\mathbb{R}^2} \to 0$.
    However, taking $w_h := (h, h^3)$ for $h>0$, we obtain
    \begin{displaymath}
        \frac{|f(w_h)|}{\left\| w_h \right\|_{\mathbb{R}^2}} 
        =
        \frac{1}{2 \sqrt{1 + h^4}}
        \to  
        \frac{1}{2}
    \end{displaymath}
    as $h \searrow 0$, a contradiction.
\end{example}

\subsection{Real-valued functions}

In this section, we focus on the case $Y = \mathbb{R}$ equipped with the absolute value norm.
For real-valued functions, the expressions in \eqref{eq:sFd} and \cref{thm:sFd_equivalence} can be rewritten in terms of $\mathcal{C}$ by \eqref{eq:ABC}.

Norms are specularly Fr\'echet differentiable at the origin.

\begin{example}
    Let $X$ be a nontrivial normed vector space over $\mathbb{R}$ equipped with a norm $\| \cdot \|_X$.
    Let $f : X \to \mathbb{R}$ be defined by 
    \begin{displaymath}
        f(x) = \left\| x \right\|_X,
    \end{displaymath}
    for $x \in X$.
    Then, it is well-known that $f$ is not Fr\'echet differentiable at $x = 0$.

    We claim that $f$ is specularly Fr\'echet differentiable at $x = 0$ with $\sFd f(0) = 0$, that is, 
    \begin{displaymath}
        \left\langle \sFd f(0), w \right\rangle = 0 
    \end{displaymath}
    for all $w \in X$.
    Indeed, for every $w\in X\setminus\{0\}$, we have
    \begin{displaymath}
        \left\| J \right\|_{X \times \mathbb{R}} = \sqrt{2} \left\| w \right\|_X = \left\| K \right\|_{X \times \mathbb{R}},
    \end{displaymath}
    where $J$ and $K$ are defined as in \eqref{def: points JK}.
    Thus, the fraction in \eqref{eq:sFd} is reduced to 
    \begin{displaymath}
        \dfrac{\left\vert \left\| w \right\|_X - \left\| 0 \right\|_X - \left\langle \sFd f(0), w \right\rangle + \left\| 0 \right\|_X - \left\| -w \right\|_X - \left\langle \sFd f(0), w \right\rangle \right\vert }{2 \left\| w \right\|_X} = 0.
    \end{displaymath}
    Taking the limit $\left\| w \right\|_X \to 0$ completes the proof of the claim. 
\end{example}

In Hilbert spaces, specular gradients admit a representation in terms of specular directional derivatives along an orthonormal basis and satisfy a bound at local extrema.

\begin{theorem} \label{thm:nec_cond_opt}
    Let $H$ be a Hilbert space over $\mathbb{R}$ equipped with a norm $\| \cdot \|_H$ and an inner product $\left\langle \, \cdot \, , \, \cdot \, \right\rangle _H$.
    Assume that a functional $f : \Omega \to \mathbb{R}$ is specularly Fr\'echet differentiable in $\Omega$, where $\Omega$ is an open subset of $H$.
    Let $\left\{ e_i \right\}_{i\in I}$ be an orthonormal basis of $H$.
    Then, for every $x\in\Omega$,
    \begin{equation}\label{eq:specular_gradient_expansion}
        \sg f(x)
        =
        \sum_{i\in I} \partial_{e_i}^{\sd}f(x)\,e_i,
    \end{equation}
    where the sum is understood as the norm limit of the finite partial sums.
    If $x^{\ast} \in \Omega$ is a local minimizer or maximizer of $f$, it holds that 
    \begin{displaymath}
        \left\vert \left\langle \sg f\left(x^{\ast}\right), v\right\rangle_H  \right\vert \leq \| v \|_H
    \end{displaymath}
    for each $v \in H$.
\end{theorem}

\begin{proof}
    By \cref{prop:spF_implies_spG} and
    \eqref{eq: specular gradient}, we have
    \begin{displaymath}
        \partial_{e_i}^{\sd}f(x)
        =
        \left\langle \sFd f(x),e_i\right\rangle
        =
        \left\langle \sg f(x),e_i\right\rangle_H
    \end{displaymath}
    for every $x\in\Omega$ and $i\in I$.
    The orthonormal basis expansion of $\sg f(x)$
    therefore gives \eqref{eq:specular_gradient_expansion}.

    If $x^{\ast}$ is a local minimizer or maximizer of $f$,
    then \cref{thm:quasi-Fermat} yields
    \begin{displaymath}
        \left|\left\langle \sg f(x^{\ast}),v\right\rangle_H\right|
        =
        \left|\partial_v^{\sd}f(x^{\ast})\right|
        \leq
        \|v\|_H
    \end{displaymath}
    for every $v\in H$.
\end{proof}

However, the existence of all the specular partial derivatives on the right-hand side of \eqref{eq:specular_gradient_expansion} does not imply specular Fr\'echet differentiability; see \cref{ex:rsd_does_not_implies_sGd}.

\bibliographystyle{siamplain}
\bibliography{reference}{}

\end{document}